\documentclass[10pt]{amsart}
\usepackage[mathscr]{eucal}
\usepackage{amsmath,amsfonts}
\newtheorem{theorem}{Theorem}[section]

\newtheorem{proposition}[theorem]{Proposition}
\newtheorem{corollary}[theorem]{Corollary}
\newtheorem{lemma}[theorem]{Lemma}
\newtheorem{example}[theorem]{Example}

\newtheorem{remark}[theorem]{Remark}

\makeatletter
\@addtoreset{equation}{section}
\makeatother

\newcommand{\CC}{{\mathbb C}}
\newcommand{\NN}{{\mathbb N}}

\newcommand{\DD}{{\mathbb D}}
\newcommand{\RR}{{\mathbb R}}
\newcommand{\FF}{{\mathbb F}}
\newcommand{\TT}{{\mathbb T}}

\newcommand{\cA}{{\mathcal A}}

\newcommand{\cD}{{\mathcal D}}
\newcommand{\cE}{{\mathcal E}}

\newcommand{\cG}{{\mathcal G}}
\newcommand{\cH}{{\mathcal H}}
\newcommand{\cK}{{\mathcal K}}

\newcommand{\cM}{{\mathcal M}}
\newcommand{\cN}{{\mathcal N}}

\newcommand{\cP}{{\mathcal P}}

\newcommand{\cR}{{\mathcal R}}
\newcommand{\cS}{{\mathcal S}}

\newcommand{\cV}{{\mathcal V}}

\newdimen\expt
\def\boxit#1{\setbox0\hbox{$\displaystyle{#1}$}
      \hbox{\lower.4\expt
 \hbox{\lower3\expt\hbox{\lower\dp0
      \hbox{\vbox{\hrule height.4\expt
 \hbox{\vrule width.4\expt\hskip3\expt
      \vbox{\vskip3\expt\box0\vskip2\expt}%
 \hskip3\expt\vrule width.4\expt}\hrule height.4\expt}}}}}}

\begin{document}




\title [ Joint similarity  to   operators in noncommutative varieties ]
{ Joint similarity  to   operators in noncommutative varieties }
 \author{Gelu Popescu}
\date{February 5, 2010}
\thanks{The author was
partially supported by an NSF grant} 

\address{Department of Mathematics, The University of Texas
at San Antonio \\ San Antonio, TX 78249, USA} \email{\tt
gelu.popescu@utsa.edu}

\begin{abstract}
In this paper we solve several  problems  concerning  joint
similarity to $n$-tuples of operators in
  noncommutative varieties  ${\cV}_{f,\cP}^m(\cH)\subset B(\cH)^n$,
 $m\geq 1$,  associated with positive regular free
holomorphic functions $f$ in $n$ noncommuting variables  and  with
sets $\cP$ of noncommutative polynomials in $n$ indeterminates,
where $B(\cH)$ is the algebra of all bounded linear operators on a
Hilbert space $\cH$. In particular, if $f=X_1+\cdots +X_n$ and
$\cP=\{0\}$, the elements of the corresponding variety can be seen
as noncommutative multivariable analogues
 of Agler's $m$-hypercontractions.

We introduce a class of generalized noncommutative Berezin transforms and use them
to solve operator inequalities associated with  noncommutative varieties ${\cV}_{f,\cP}^m(\cH)$.
We point out a very strong connection between the cone of  their positive solutions
 and the  joint similarity problems. Several classical results concerning  the similarity
 to contractions have  analogues  in our noncommutative multivariable setting.
 When $\cP$ consists of the commutators $X_iX_j-X_jX_i$, \, $i,j\in \{1,\ldots,n\}$, we obtain commutative versions of these results.
  We remark that, in the  particular case when $n=m=1$, $f=X$, and $\cP=\{0\}$, we recover  the corresponding
similarity results obtained by Sz.-Nagy, Rota, Foia\c s, de Branges-Rovnyak, and Douglas.

 We use some of the results of this paper  to  provide Wold type decompositions and   triangulations for
 $n$-tuples of operators in  noncommutative varieties
${\cV}_{f,\cP}^1(\cH)$, which parallel the  classical Sz.-Nagy--Foia\c s triangulations
 for contractions  but also provide new proofs.
As consequences, we prove the existence of  joint invariant subspaces for certain
 classes of operators in ${\cV}_{f,\cP}^1(\cH)$.
\end{abstract}

\maketitle

\section*{Introduction}

Let $B(\cH)$ denote  the algebra of all bounded linear operators on
a Hilbert space $\cH$. Two operators $A,B\in B(\cH)$ are called similar
 if there is an invertible operator $S\in B(\cH)$ such that $A=S^{-1} B S$.
 The problem of characterizing
the operators similar to contractions, i.e., the operators in the unit ball
$$
[B(\cH)]_1:=\{X\in B(\cH):\ XX^*\leq I\},
$$
or similar to special   contractions such as parts of shifts,
isometries, unitaries, etc., has been considered  by many authors
and has generated deep results in operator theory and operator
algebras. We shall  mention some of the classical  results on
similarity that strongly influenced us in writing this paper.

In 1947, Sz.-Nagy \cite{SzN} found necessary and sufficient conditions
for an operator to be similar to a unitary operator. In particular,  an
operator $T$ is
similar to an isometry if and only if there are constants $a,b>0$ such that
$$
a\|h\|\leq  \|T^nh\|\leq b\|h\|,\qquad  h\in \cH, n\in \NN.
$$
The fact that the unilateral shift on the Hardy space $H^2(\TT)$   plays  the role of {\it universal model} in $B(\cH)$
was discovered by Rota  \cite{R}.
Rota's model theorem  asserts that any operator with spectral radius less than
 one is similar to a contraction, or more precisely, to a part of a backward
 unilateral shift. This result was refined furthermore by Foia\c s \cite{Fo}
  and  by de Branges and Rovnyak \cite{BR}, who proved that every strongly
  stable
   contraction is unitarily equivalent to a part of a backward unilateral shift.

It is well-known that if   $T\in B(\cH)$ is similar to a contraction then, due to the von
Neumann inequality \cite{von}, it  is polynomially bounded, i.e., there is a
constant $C>0$ such that, for any polynomial $p$,
$$
\|p(T)\|\leq C \|p\|_\infty,
$$
 where $ \|p\|_\infty:= \sup_{|z|=1} |p(z)|$. A remarkable result  obtained
 by Paulsen \cite{Pa} shows that similarity  to a contraction is equivalent
 to complete polynomial boundedness. Halmos' famous similarity problem \cite{H2}
   asked  whether any polynomially bounded operator is similar to a contraction.
   This long standing problem was answered  by Pisier
 \cite{Pi} in  a remarkable paper where  he shows that there are  polynomially
  bounded operators which are  not similar to  contractions. For more information
  on similarity problems and completely bounded maps we refer the reader to the
   excellent books by Pisier \cite{Pi-book} and Paulsen \cite{Pa-book}.

In the noncommutative multivariable setting, joint similarity problems to
row contractions, i.e., $n$-tuples of operators in the unit ball
$$
[B(\cH)^n]_1:=\{(X_1,\ldots, X_n)\in B(\cH)^n: \ X_1X_1^*+\cdots
+X_nX_n^*\leq I\},
$$
were  considered by Bunce \cite{Bu}, the author (see \cite{Po-models},
\cite{Po-similarity}, \cite{Po-charact2}, \cite{Po-unitary}), and recently by Douglas,
Foia\c s, and Sarkar \cite{DFS}.
In this setting, the {\it universal model} for  the unit ball  $[B(\cH)^n]_1$ is
 the $n$-tuple $(S_1,\ldots, S_n)$ of left creation operators on the full Fock space with $n$ generators.

To put our work in perspective we need some notation.
Let  $\FF_n^+$ be  the unital free semigroup on $n$ generators
$g_1,\ldots, g_n$ and the identity $g_0$.  The length of $\alpha\in
\FF_n^+$ is defined by $|\alpha|:=0$ if $\alpha=g_0$  and
$|\alpha|:=k$ if
 $\alpha=g_{i_1}\cdots g_{i_k}$, where $i_1,\ldots, i_k\in \{1,\ldots, n\}$.
If $X:=(X_1,\ldots, X_n)\in B(\cH)^n$   we
denote $X_\alpha:= X_{i_1}\cdots X_{i_k}$  and $X_{g_0}:=I_\cH$, the identity on $\cH$.

In \cite{Po-domains} (case $m=1$) and \cite{Po-Berezin} (case $m\geq 2$),
 we studied  more general noncommutative domains
$$
{\bf D}_p^m(\cH):=\left\{X:= (X_1,\ldots, X_n)\in B(\cH)^n: \
 (id-\Phi_{p,X})^s(I)\geq 0 \ \text{ for } \ s=1,\ldots, m\right\},
$$
where $id$ is the identity map on $B(\cH)$,
$$\Phi_{p,X}(Y):=\sum_{|\alpha|\geq 1} a_\alpha X_\alpha YX_\alpha^*,\qquad Y\in B(\cH),
$$
  and  $p=\sum_{|\alpha|\geq 1} a_\alpha X_\alpha$ is a
   positive regular noncommutative polynomial, i.e., its coefficients are positive
    scalars and $a_\alpha>0$ if  $\alpha\in \FF_n^+$ with $|\alpha|=1$.
We  remark that
if $q=X_1+\cdots+X_n$ and $m\geq 1$,  then ${\bold D}_q^m(\cH)$  is a starlike domain
 which concides with  the set of all row contractions $(X_1,\ldots, X_n)\in [B(\cH)^n]_1$
satisfying the positivity condition
$$
\sum_{k=0}^m (-1)^k \left(\begin{matrix} m\\k\end{matrix}\right)\sum_{|\alpha|=k} X_\alpha X_\alpha^*\geq 0.
$$
The elements of  the domain ${\bold D}_q^m(\cH)$  can be seen as
multivariable noncommutative analogues of Agler's
$m$-hypercontractions \cite{Ag2}. The case $n=1$ was  recently
studied by Olofsson (\cite{O2}, \cite{O1}). We showed
(\cite{Po-domains}, \cite{Po-Berezin}) that each domain ${\bf
D}_p^m(\cH)$ has
 a  {\it universal model} $(W_1,\ldots, W_n)$ of
{\it weighted left creation operators} \, acting on the full Fock space   with $n$
generators. The  study of the domain ${\bf D}_p^m(\cH)$  and the dilation theory associated with it are  close related to
the study of the weighted shifts $W_1,\ldots,W_n$, their joint
invariant subspaces, and the representations of the algebras they
generate: the domain algebra $\cA_n({\bf D}_p^m)$, the Hardy algebra
$F_n^\infty({\bf D}_p^m)$, and the $C^*$-algebra $C^*(W_1,\ldots,
W_n)$.

In the present paper, we consider   problems of joint similarity to classes of
$n$-tuples of operators in  noncommutative domains ${\bold D}_p^m(\cH)$, $m\geq 1$,  and
 noncommutative varieties
 $$
\cV_{p,\cP}^m(\cH):= \left\{(X_1,\ldots, X_n)\in {\bf D}_p^m(\cH):\
q(X_1,\ldots, X_n)=0\quad \text{ for any }\quad q\in \cP\right\},
$$
where $\cP$ is a family of  noncommutative  polynomials in $n$
indeterminates.

In Section 1, expanding on the author's work (\cite{Po-Berezin},
\cite{Po-pluriharmonic}, \cite{Po-domains}) on noncommutative
Berezin transforms, we introduce a  new class of generalized Berezin
transforms which will play an important role in this paper. Given
$A:=(A_1,\ldots, A_n)\in B(\cH)^n$, our similarity problems to
$n$-tuples of operators in the noncommutative variety
$\cV^m_{p,\cP}(\cH)$ are linked to the noncommutative cone
$C(p,A)^+$ of all positive operators $D\in B(\cH)$ such that
$$
(id-\Phi_{p,A})^s(D)\geq 0,\qquad s=1,\ldots,m.
$$
For example, $(A_1,\ldots, A_n)$ is jointly similar to an $n$-tuple
of operators in $\cV^m_{p,\cP}(\cH)$ if and only if there is an
invertible operator in $C(p,A)^+$. Under natural conditions, we show
that there is a one-to-one correspondence between the elements of the  noncommutative
cone $C(p,A)^+$ and  a class of generalized Berezin transforms, to
be introduced.

In Section 2, a  pure version of the above-mentioned  result is
established, even in a more general setting. In particular, when
$m=1$ and $T:=(T_1,\ldots, T_n)\in \cV^1_{p,\cP}(\cH)$ is {\it pure},
i.e., $\Phi_{p,T}^k(I)\to 0$ strongly, as $k\to\infty$, we determine
the  noncommutative cone $C(p,T)^+$ by  showing that all its
elements have the form $P_\cH \Psi \Psi^*|_{\cH}$, where $\Psi$ is a
multi-analytic operator with respect to the universal $n$-tuple
$(B_1,\ldots, B_n)$ associated with the variety $\cV^1_{p,\cP}(\cH)$.
More precisely, $\Psi\in R_n^\infty(\cV^1_{p,\cP})\bar\otimes
B(\cK,\cK')$ for some Hilbert spaces $\cK$ and $\cK'$, where
$R_n^\infty(\cV^1_{p,\cP})$ is the commutant  of the noncommutative Hardy algebra
$F_n^\infty(\cV^1_{p,\cP})$. We remark that in the  particular case
when $n=m=1$, $p=X$,  $\cP=\{0\}$,
      and $\Phi_{p,T}(X):=TXT^*$ with $\|T\|\leq1$,
       the corresponding cone $C(p,T)^+$
      was studied by Douglas in \cite{Do} and by Sz.-Nagy
      and Foia\c s \cite{SzF1}
      in connection with $T$-Toeplitz operators (see also
  \cite{CF1} and \cite{CF2}).

 In Section 3,   we
 provide necessary and sufficient conditions
for  an $n$-tuple   $A:=(A_1,\ldots, A_n)\in B(\cH)^n$ to be jointly
similar to an $n$-tuple  of operators $T:=(T_1,\ldots, T_n)$ in the
noncommutative variety  ${\cV}_{p,\cP}^m(\cH)$ or   the distinguished  sets
$$\left\{ X \in {\cV}_{p,\cP}^m(\cH): \
 (id-\Phi_{p,X})^m(I)= 0\right\}\quad \text{ and }\quad
 \left\{ X \in {\cV}_{p,\cP}^m(\cH): \
 (id-\Phi_{p,X})^m(I)>0\right\},
$$
where $\cP$ is a set of noncommutative polynomials.   To give the
reader a flavor of  our  results, we shall be a little  bit more precise.
Given $(A_1,\ldots, A_n)\in B(\cH)^n$, we find  necessary and sufficient conditions
for the existence of an invertible operator $Y:
\cH\to \cG$ such that
$$
A_i^*=Y^{-1}[(B_i^*\otimes I_\cH)|_\cG]Y,\qquad i=1,\ldots, n
$$
where
 $\cG\subseteq \cN_\cP\otimes \cH$ is an invariant  subspace under  each operator $B_i^*\otimes
 I_\cH$  and $(B_1,\ldots, B_n)$ is
the universal
  model associated with the noncommutative variety
 $\cV^m_{f,\cP}(\cH)$. In particular, we  obtain  an analogue of Foia\c s \cite{Fo} and de Branges--Rovnyak \cite{BR}  model theorem, for pure $n$-tuples of operators in $\cV^m_{f,\cP}(\cH)$. We also obtain   the following Rota type \cite{R}
 model  theorem for    the noncommutative variety
$\cV^m_{f,\cP}(\cH)$. If
$A:=(A_1,\ldots, A_n)\in B(\cH)^n$ is  such that $q(A_1,\ldots,
A_n)=0$ for $q\in \cP$ and
$$
 \sum_{k=0}^\infty \left(\begin{matrix} k+m-1\\ m-1
\end{matrix}\right)\Phi_{p,A}^k(I)\leq bI
$$
for some constant  $ b>0$, then the above-mentioned joint similarity  holds.
Moreover, we prove that  the joint spectral radius
 $r_p(A_1,\ldots,A_n)<1$   if and only if $(A_1,\ldots, A_n)$ is jointly similar to an $n$-tuple
 $T:=(T_1,\ldots, T_n) \in {\cV}_{p,\cP}^m(\cH)$ with
 $(id-\Phi_{p,T})^m(I)>0$, i.e., positive invertible operator.

  We also
 provide necessary and sufficient conditions
for  an $n$-tuple   $A:=(A_1,\ldots, A_n)\in B(\cH)^n$ to be jointly
similar to an $n$-tuple  of operators $T:=(T_1,\ldots, T_n)\in
{\cV}_{p,\cP}^m(\cH)$  with  $ (id-\Phi_{p,T})^m(I)= 0$. Our
noncommutative  analugue of  Sz.-Nagy's similarity result \cite{SzN}
asserts that there is  an invertible operator $Y\in B(\cH)$ such
that $ A_i=Y^{-1} T_i Y$, $i=1,\ldots,n,$ if and only if there exist
 positive constants
 $0<c\leq d$ such that
 $$
 cI\leq  \Phi_{p,A}^k(I)\leq dI, \qquad k\in \NN.
 $$
In particular, we obtain a  multivariable analogue  of Douglas' similarity result \cite{Do}.

If $(A_1,\ldots, A_n)\in B(\cH)^n$ is  jointly
similar to an $n$-tuple  of operators in a {\it radial noncommutative variety}
${\cV}_{p,\cP}^m(\cH)$,  where $\cP$ is a set of homogeneous noncommutative polynomials, then the polynomial calculus
$g(B_1,\ldots, B_n)\mapsto g(A_1,\ldots, A_n)
$
 can be extended to a completely bounded  map on the noncommutative variety algebra $\cA_n({\cV}_{p,\cP}^m)$, the norm closed algebra generated by $B_1,\ldots, B_n$ and the identity. Using Paulsen's similarity result \cite{Pa}, we  can prove that the converse is true if $m=1$,  but remains an open problem if $m\geq 2$.

In Section 4, we obtain Wold type decompositions and  prove the existence of  triangulations of type
$$
\left(\begin{matrix}C_{\cdot 0}&0\\
*& C_{\cdot 1}\end{matrix} \right)\quad \text{ and }\quad \left(\begin{matrix}C_{c}&0\\
*& C_{cnc}\end{matrix} \right)
$$
for any $n$-tuple  of operators in the noncommutative variety
${\cV}^1_{p,\cP}(\cH)$, which parallel the   Sz.-Nagy--Foia\c s
\cite{SzF-book}  triangulations
 for contractions. The proofs seem to be new even in the classical case $n=1$, since they don't involve, at least explicitly, the dilation space for contractions.
As consequences, we prove the existence of  joint invariant
subspaces for certain
 classes of operators in ${\cV}^1_{p,\cP}(\cH)$.

We should mention that the results of this paper are
presented in a
 more general setting  when the polynomials $p$  in the definition of ${\cV}_{p,\cP}^m(\cH)$  is
  replaced by positive regular free holomorphic functions.

\bigskip

\section{Generalized noncommutative Berezin transforms  and the cone $C(f,A)^+$}\label{Berezin}

In this section, we introduce a class of generalized
Berezin transforms which will play an important role in this paper. We use them to study
the noncommutative cone $C(f,A)^+$
of all positive solutions of the operator inequalities
$$
(id-\Phi_{f,A})^s(X)\geq 0,\qquad s=1,\ldots,m.
$$

First, we recall (\cite{Po-Berezin}, \cite{Po-domains}) the
construction of the universal model associated with   the
noncommutative domain ${\bold D}_f^m(\cH)$, $m\geq 1$.
 Throughout this paper, we
  assume that $f:= \sum_{\alpha\in
\FF_n^+} a_\alpha X_\alpha$, \ $a_\alpha\in \CC$,  is  a {\it
positive regular free holomorphic function} in $n$ variables
$X_1,\ldots,X_n$. This means
\begin{enumerate}
\item[(i)]
 $
\limsup_{k\to\infty} \left( \sum_{|\alpha|=k}
|a_\alpha|^2\right)^{1/2k}<\infty$,
\item[(ii)]  $a_\alpha\geq 0$ for
any $\alpha\in \FF_n^+$, \ $a_{g_0}=0$,
 \ and  $a_{g_i}>0$ for  $i=1,\ldots, n$.
  \end{enumerate}
Given $m\in \NN:=\{1,2,\ldots\}$ and  a positive regular free
holomorphic
  function $f$ as above,
  we define  the noncommutative domain ${\mathbf D}_f^m$ whose
  representation on a Hilbert space $\cH$ is
$$
{\mathbf D}_f^m(\cH):=\left\{X:= (X_1,\ldots, X_n)\in B(\cH)^n: \
 (id-\Phi_{f,X})^s(I)\geq 0  \ \text{ for } \ s=1,\ldots,  m\right\},
$$
where $\Phi_{f,X}:B(\cH)\to B(\cH)$ is given by $$
\Phi_{f,X}(Y):=\sum_{k=1}^\infty\sum_{|\alpha|=k} a_\alpha X_\alpha
YX_\alpha^*, \qquad   Y\in B(\cH), $$
 and the convergence is in the weak
operator topology.  ${\bold D}_f^m(\cH)$ can be seen as a
noncommutative Reinhardt domain, i.e., $(e^{i\theta_1}X_1,\ldots,
e^{i\theta_n}X_n)\in {\bold D}_f^m(\cH)$ for any $(X_1,\ldots,
X_n)\in {\bold D}_f^m(\cH)$ and $\theta_1,\ldots, \theta_n\in \RR$.

  Let $H_n$ be
an $n$-dimensional complex  Hilbert space with orthonormal basis
$e_1,\dots,e_n$, where $n\in \NN$ or $n=\infty$.
  We consider the full Fock space  of $H_n$ defined by
$$F^2(H_n):=\bigoplus_{k\geq 0} H_n^{\otimes k},$$
where $H_n^{\otimes 0}:=\CC 1$ and $H_n^{\otimes k}$ is the
(Hilbert) tensor product of $k$ copies of $H_n$. Set $e_\alpha :=
e_{i_1}\otimes e_{i_2}\otimes \cdots \otimes e_{i_k}$ if
$\alpha=g_{i_1}g_{i_2}\cdots g_{i_k}\in \FF_n^+$
 and $e_{g_0}:= 1$.
It is  clear that $\{e_\alpha:\alpha\in\FF_n^+\}$ is an orthonormal
basis of $F^2(H_n)$.
Define the left creation operators $S_i:F^2(H_n)\to F^2(H_n), \
i=1,\dots, n$,  by
$
 S_i f:=e_i\otimes f , \ f\in F^2(H_n).
$

Let $D_i:F^2(H_n)\to F^2(H_n)$, $i=1,\ldots, n$,  be the diagonal
operators  given by
$$
D_ie_\alpha:=\sqrt{\frac{b_\alpha^{(m)}}{b_{g_i \alpha}^{(m)}}}
e_\alpha,\qquad
 \alpha\in \FF_n^+,
$$
where
\begin{equation}
\label{b-al} b_{g_0}^{(m)}:=1\quad \text{ and } \quad
 b_\alpha^{(m)}:= \sum_{j=1}^{|\alpha|}
\sum_{{\gamma_1\cdots \gamma_j=\alpha }\atop {|\gamma_1|\geq
1,\ldots, |\gamma_j|\geq 1}} a_{\gamma_1}\cdots a_{\gamma_j}
\left(\begin{matrix} j+m-1\\m-1
\end{matrix}\right)  \qquad
\text{ if } \ \alpha\in \FF_n^+, \,|\alpha|\geq 1.
\end{equation}
  We have
$$
\|D_i\|=\sup_{\alpha\in \FF_n^+} \sqrt{\frac{b_\alpha^{(m)}}{b_{g_i
\alpha}^{(m)}}}\leq \frac{1}{\sqrt{a_{g_i}}}, \qquad i=1,\ldots,n.
$$
Define the {\it weighted left creation  operators} $W_i:F^2(H_n)\to
F^2(H_n)$, $i=1,\ldots, n$,  associated with the noncommutative
domain  ${\bold D}_f^m $  by setting $W_i:=S_iD_i$, where
 $S_1,\ldots, S_n$ are the left creation operators on the full
 Fock space $F^2(H_n)$.
 Note that
\begin{equation*}
W_i e_\alpha=\frac {\sqrt{b_\alpha^{(m)}}}{\sqrt{b_{g_i
\alpha}^{(m)}}} e_{g_i \alpha}, \qquad \alpha\in \FF_n^+.
\end{equation*}
 One can easily see that
\begin{equation}\label{WbWb}
W_\beta e_\gamma= \frac {\sqrt{b_\gamma^{(m)}}}{\sqrt{b_{\beta
\gamma}^{(m)}}} e_{\beta \gamma} \quad \text{ and }\quad W_\beta^*
e_\alpha =\begin{cases} \frac
{\sqrt{b_\gamma^{(m)}}}{\sqrt{b_{\alpha}^{(m)}}}e_\gamma& \text{ if
}
\alpha=\beta\gamma \\
0& \text{ otherwise }
\end{cases}
\end{equation}
 for any $\alpha, \beta \in \FF_n^+$.
According to Theorem 1.3 from \cite{Po-Berezin}, the weighted left
creation
  operators $W_1,\ldots, W_n$     associated with ${\bold D}_f^m $  have the following properties:
 \begin{enumerate}
 \item[(i)] $\sum_{k=1}^\infty\sum_{|\beta|=k} a_\beta W_\beta W_\beta^*\leq I$, where the
convergence is in the strong operator topology;
 \item[(ii)]
 $\left(id-\Phi_{f,W}\right)^{m}(I)=P_\CC$, where $P_\CC$ is the
 orthogonal projection from $F^2(H_n)$ onto $\CC 1\subset F^2(H_n)$, and
  $\lim\limits_{p\to\infty} \Phi^p_{f,W}(I)=0$ in the strong operator
 topology.
 \end{enumerate}
     The $n$-tuple $(W_1,\ldots,
W_n)\in {\bf D}_f^m (F^2(H_n))$ plays the role of universal model
for the noncommutative domain ${\bf D}_f^m $. The domain algebra
$\cA_n({\bf D}^m_f)$ associated with the noncommutative domain ${\bf
D}^m_f$ is the norm closure of all polynomials in $W_1,\ldots, W_n$,
and the identity, while the Hardy algebra  $F_n^\infty({\bf D}_f^m)$
is the SOT-(WOT-, or $w^*$-)  version.

We remark that, one  can also define the {\it weighted right
creation operators} $\Lambda_i:F^2(H_n)\to F^2(H_n)$ by setting
$\Lambda_i:= R_i G_i$, $i=1,\ldots,n$,  where $R_1,\ldots, R_n$ are
 the right creation operators on the full Fock space $F^2(H_n)$ and
 each  diagonal operator $G_i$  is defined by
$$
G_ie_\alpha:=\sqrt{\frac{b_\alpha^{(m)}}{b_{ \alpha g_i}^{(m)}}}
e_\alpha,\quad
 \alpha\in \FF_n^+,
$$
where the coefficients $b_\alpha^{(m)}$, $\alpha\in \FF_n^+$, are
given by relation \eqref{b-al}. It turns out that
  $(\Lambda_1 ,\ldots,
\Lambda_n )$ is in the noncommutative domain $ {\bf D}_{\widetilde
f}^{m}(F^2(H_n)),$ where
  $\widetilde{f}:=\sum_{|\alpha|\geq 1} a_{\widetilde \alpha} X_\alpha$ and
$\widetilde \alpha=g_{i_k}\cdots g_{i_1}$ denotes the reverse of
$\alpha=g_{i_1}\cdots g_{i_k}\in \FF_n^+$.
 Moreover,   $W_i \Lambda_j =\Lambda_j  W_i $
   and $U^* \Lambda_i  U=W_i $,
$i=1,\ldots,n$, where $U\in B(F^2(H_n))$ is the unitary operator
defined by equation $U e_\alpha:= e_{\widetilde \alpha}$, $\alpha\in
\FF_n^+$. Consequently, we have
\begin{equation*}
F_n^\infty({\bf D}_f^m)^\prime=R_n^\infty({\bf D}_f^m)\ \text{ and }
\ R_n^\infty({\bf D}_f^m)^\prime=F_n^\infty({\bf D}_f^m),
\end{equation*}
where $^\prime$ stands for the commutant and $R_n^\infty({\bf
D}_f^m)$  is the   SOT-(WOT-, or $w^*$-) closure of all polynomials
in $\Lambda_1,\ldots, \Lambda_n$, and the identity. More on these
noncommutative Hardy algebras can be found in \cite{Po-von},
\cite{Po-Berezin}, and \cite{Po-domains}.

 \bigskip

In what follows, we introduce a noncommutative Berezin kernel
associated with any quadruple $(f,m, A,R)$ satisfying the following
compatibility conditions:
\begin{enumerate}
\item[(i)]
  $f:=\sum_{|\alpha|\geq 1} a_\alpha X_\alpha$ is
a positive regular free holomorphic function and $m\in \NN$;
\item[(ii)]
   $A:=(A_1,\ldots, A_n)\in B(\cH)^n$ is   such that
$\sum\limits_{|\alpha|\geq 1} a_\alpha A_\alpha A_\alpha^* $ \ is
SOT-convergent;
  \item[(iii)] $R\in
B(\cH)$ is a positive operator such that
 $$
\sum_{k=0}^\infty \left(\begin{matrix} k+m-1\\ m-1
\end{matrix}\right)\Phi_{f,A}^k(R)\leq b I,
$$
 for some constant $b>0$.
\end{enumerate}
The noncommutative Berezin kernel associated with the compatible
quadruple $(f, m, A,R)$ is the operator $K^{(m)}_{f,A,R}:\cH\to
F^2(H_n)\otimes \overline{ R^{1/2}(\cH)}$  \, given by
\begin{equation}
\label{Po-ker}
 K^{(m)}_{f,A,R}h=\sum_{\alpha\in \FF_n^+} \sqrt{b_\alpha^{(m)}}
e_\alpha\otimes   R^{1/2} A_\alpha^* h,\qquad h\in \cH.
\end{equation}

\begin{lemma}\label{lemma1}     The noncommutative
Berezin  kernel \ $K^{(m)}_{f,A,R}$ associated with a compatible
quadruple  $(f,m, A,R)$ is a bounded operator and
\begin{equation*}
  K^{(m)}_{f,A,R} A_i^*=(W_i^*\otimes
I_{\cR})K^{(m)}_{f,A,R},\qquad i=1,\ldots, n,
\end{equation*}
where  $\cR:=\overline{ R^{1/2}(\cH)}$ and $(W_1,\ldots, W_n)$ is
the universal model associated with
the noncommutative domain ${\bf D}_f^m $. Moreover,
\begin{equation*}
  \left(K^{(m)}_{f,A,R}\right)^* K^{(m)}_{f,A,R}= \sum_{k=0}^\infty \left(\begin{matrix} k+m-1\\ m-1
\end{matrix}\right)
\Phi_{f,A}^k(R).
\end{equation*}
\end{lemma}

\begin{proof} Since $(f, m,
A,R)$ is a compatible quadruple, $R\in B(\cH)$ is a positive
operator such that
 \begin{equation}
 \label{bound}
\sum_{k=0}^\infty \left(\begin{matrix} k+m-1\\ m-1
\end{matrix}\right)\Phi_{f,A}^k(R)\leq b I
\end{equation}
 for some constant $b>0$.
Note that due to relations \eqref{b-al} and \eqref{Po-ker}, we
 have
\begin{equation*}
\begin{split}
\|K^{(m)}_{f,A,R}h\|^2&= \sum_{\beta\in \FF_n^+} b_\beta^{(m)}\left<
A_\beta RA_\beta^*h, h\right>
 =
\left< R h,h\right>+ \sum_{m=1}^\infty \sum_{|\beta|=m}\left<
b_\beta^{(m)} A_\beta  RA_\beta^*h, h\right>\\
& =\left< R h,h\right>+ \sum_{m=1}^\infty
\sum_{|\beta|=m}\left<\left(
\sum_{j=1}^{|\beta|}\left(\begin{matrix} j+m-1\\ m-1
\end{matrix}\right)\sum_{{\gamma_1\cdots
\gamma_j=\beta}\atop{|\gamma_1|\geq 1,\ldots, |\gamma_j|\geq 1}}
a_{\gamma_1}\cdots a_{\gamma_j}\right) A_{\gamma_1\cdots \gamma_j}
 R A_{\gamma_1\cdots \gamma_j}^*h,h\right>\\
& =\left< R h,h\right>+ \sum_{k=1}^\infty \left<\left(\begin{matrix}
k+m-1\\ m-1
\end{matrix}\right)\Phi_{f,A}^k(
R)h,h\right>
\end{split}
\end{equation*}
for any $h\in \cH$.  Hence and due to  relation \eqref{bound}, we
deduce that $K^{(m)}_{f,A,R}$ is  a well-defined bounded operator
and
\begin{equation*}
  \left(K^{(m)}_{f,A,R}\right)^* K^{(m)}_{f,A,R}= \sum_{k=0}^\infty
\left(\begin{matrix} k+m-1\\ m-1
\end{matrix}\right)\Phi_{f,A}^k(R).
\end{equation*}
 On the other hand,  due to relations \eqref{Po-ker} and \eqref{WbWb}, we have
\begin{equation*}
\begin{split}
(W_i^*\otimes I_\cR)K^{(m)}_{f,A,R}h&= \sum_{\alpha\in \FF_n^+}
\sqrt{b_\alpha^{(m)}} W_i^*e_\alpha\otimes   R^{1/2} A_\alpha^* h\\
&=\sum_{\gamma\in \FF_n^+}
\sqrt{b^{(m)}_{g_i \gamma}} W_i^*e_{g_i \gamma}\otimes   R^{1/2} A_{g_i \gamma}^* h\\
&=\sum_{\gamma\in \FF_n^+}
\sqrt{b^{(m)}_{\gamma}}  e_{\gamma}\otimes   R^{1/2} A_{  \gamma}^* A_i^*h\\
&=K^{(m)}_{f,A,R} A_i^*h
\end{split}
\end{equation*}
for any $h\in \cH$.
 Hence,
\begin{equation*}
  K^{(m)}_{f,A,R} A_i^*=(W_i^*\otimes
I_\cR)K^{(m)}_{f,A,R},\qquad i=1,\ldots, n,
\end{equation*}
and the proof is complete.
  \end{proof}

Let $f:=\sum_{|\alpha|\geq1} a_\alpha X_\alpha$ be a positive
regular free holomorphic function and
 let $W_1,\ldots, W_n$  and $\Lambda_1,\ldots \Lambda_n$  be   the weighted left  and right  creation operators, respectively,
 associated with  the noncommutative domain ${\bf D}_f^m$.
 Let $\cP$ be a family of  noncommutative  polynomials   and define the noncommutative
 variety $\cV_{f,\cP}^m$ whose
  representation on a Hilbert space $\cH$ is
 $$
\cV_{f,\cP}^m(\cH):= \left\{(X_1,\ldots, X_n)\in {\bf D}_f^m(\cH):\
p(X_1,\ldots, X_n)=0\quad \text{ for any }\quad p\in \cP\right\}.
$$
We associate with $\cV_{f,\cP}^m$ the    operators $B_1,\ldots, B_n$
defined as follows. Consider  the subspaces
 \begin{equation*}
 \cM_\cP:=\overline{\text{\rm span}}\{W_\alpha p(W_1,\ldots, W_n) W_\beta(1): \ p\in \cP,
 \alpha, \beta\in \FF_n^+\}
 \end{equation*}
 and $\cN_\cP:=F^2(H_n)\ominus \cM_\cP$. Throughout this paper, unless otherwise specified,
  we assume that $\cN_\cP\neq
 \{0\}$.
  It is easy to see that $\cN_\cP$ is invariant under each
 operator $W_1^*,\ldots, W_n^*$ and $\Lambda^*_1,\ldots,
 \Lambda^*_n$.
  Define
 $$
 B_i:=P_{\cN_\cP}
 W_i|_{\cN_\cP}\quad \text{  and  }\quad C_i:=P_{\cN_\cP} \Lambda_i|_{\cN_\cP},\qquad i=1,\ldots, n,
 $$
  where
 $P_{\cN_\cP}$ is the orthogonal projection of $F^2(H_n)$ onto $\cN_{\cP}$.

  The $n$-tuple of operators
$B:=(B_1,\ldots, B_n)\in \cV_{f,\cP}^m(\cN_\cP)$ plays the role of
universal model
   for the noncommutative variety $\cV_{f,\cP}^m$.
The noncommutative variety algebra $\cA_n(\cV^m_{f,\cP})$ is the
norm-closed algebra generated by $B_1,\ldots, B_n$ and the identity,
while the Hardy algebra  $F_n^\infty(\cV^m_{f,\cP})$ is the
  $w^*$-version. More on these   Hardy algebras
associated with noncommutative varieties can be found in
\cite{Po-domains} and \cite{Po-Berezin}.

Let $(f, m, A, R)$ be a compatible quadruple. Assume that the
$n$-tuple $A:=(A_1,\ldots, A_n)\in B(\cH)^n$ has, in addition, the
property that
$$
p(A_1,\ldots, A_n)=0,\qquad p\in \cP.
$$
 Under these conditions,  the tuple  $q:=(f, m, A,R,\cP)$ is called compatible.
  We define the {\it (constrained) noncommutative  Berezin kernel} associated with the  tuple
  $q$ to be
the operator $K_q:\cH\to \cN_\cP\otimes \overline{ R^{1/2}(\cH)}$
given by
$$K_q:=(P_{\cN_{\cP}}\otimes I_{\overline{R^{1/2}(\cH)}})K^{(m)}_{f,A,R},
$$
where $K^{(m)}_{f,A,R}$ is the Berezin kernel associated with the
quadruple $(f,m, A,R)$  and defined by relation \eqref{Po-ker}.

\begin{lemma}\label{lemma2} Let   $K_q$ be  the noncommutative
Berezin kernel  associated with a compatible tuple
$q:=(f,m,A,R,\cP)$.  Then
\begin{equation*}
  K_q A_i^*=(B_i^*\otimes
I_\cR)K_q,\qquad i=1,\ldots, n,
\end{equation*}
where  $\cR:=\overline{ R^{1/2}(\cH)}$ and $(B_1,\ldots, B_n)$ is
the universal model  associated with the noncommutative variety
$\cV^m_{f,\cP}$. Moreover,
\begin{equation*}
  K_q^* K_q= \sum_{k=0}^\infty \left(\begin{matrix} k+m-1\\ m-1
\end{matrix}\right)
\Phi_{f,A}^k(R).
\end{equation*}
\end{lemma}

\begin{proof}
Using  Lemma \ref{lemma1}  and the fact that $p(A_1,\ldots,A_n)=0$
for all $p\in \cP$, we obtain
$$
\left<  K^{(m)}_{f,A,R}x, [W_\alpha p(W_1,\ldots, W_n)
W_\beta(1)]\otimes y\right>=\left<x,A_\alpha
p(A_1,\ldots,A_n)A_\beta
  (K^{(m)}_{f,A,R})^*(1\otimes y)\right>=0
$$
for any $x\in \cH$, $y\in \overline{ R^{1/2}(\cH)}$, and $p\in \cP$.
 Hence, we deduce  that
\begin{equation}\label{range}
\text{\rm range}\,K^{(m)}_{f,A,R}\subseteq \cN_\cP\otimes \overline{
R^{1/2}(\cH)}.
\end{equation}
Taking into account the definition of  the  constrained  Berezin
kernel $K_q:\cH\to \cN_\cP\otimes \overline{ R^{1/2}(\cH)}$,  one
can use Lemma \ref{lemma1} and relation \eqref{range} to complete
the proof.
\end{proof}

We introduce now  the  {\it noncommutative Berezin transform} ${\bf
B}_q$
 associated with the compatible  tuple $q:=(f,m,A,R,\cP)$ to be the operator
 ${\bf B}_q: B(\cN_\cP)\to B(\cH)$ given by
 $$
 {\bf B}_q[\chi]:=K_q^*[\chi\otimes I_\cR] K_q,\qquad \chi\in B(\cN_\cP).
 $$
 where  $\cR:=\overline{ R^{1/2}(\cH)}$.
 This transform will play an important role in this paper.
To  justify the terminology,  we shall  consider the particular case
when  the  $n$-tuple  $A:=(A_1,\ldots, A_n)$ has the joint spectral
radius
$$
r_f(A_1,\ldots, A_n):=\lim_{k\to\infty}\|\Phi_{f,A}^k(I)\|^{1/2k}<1.
 $$
 Then, as in the  particular case  considered in \cite{Po-Berezin}, one can show that
\begin{equation*}
 \left<{\bf B}_q[\chi]x,y\right>=
\left<\left( I-\sum_{|\alpha|\geq 1} {a}_{\widetilde\alpha}
C_\alpha^* \otimes A_{\widetilde\alpha} \right)^{-m}
 (\chi\otimes R)
  \left(
I-\sum_{|\alpha|\geq 1} a_{\widetilde\alpha} C_\alpha \otimes
A_{\widetilde\alpha}^* \right)^{-m}(1\otimes x), 1\otimes y\right>
\end{equation*}
  for any  $x,y\in\cH$, where $C_i:=P_{\cN_\cP} \Lambda_i|_{\cN_\cP }$ for $i=1,\ldots, n$ and $\widetilde \alpha$ is the reverse of $\alpha\in \FF_n^+$.
    For the benefit of the reader, we present a sketch of the proof.
  First, one can show that
  \begin{equation*}
  r\left(\sum_{|\alpha|\geq 1} a_{\widetilde\alpha}
C_\alpha \otimes A_{\widetilde\alpha}^*\right)\leq r_f(A_1,\ldots,
A_n)<1,
\end{equation*}
where $r(Y)$ is the usual spectral radius of a bounded operator $Y$.
Hence, the operator
\begin{equation*}
 \left(
I-\sum_{|\alpha|\geq 1} a_{\widetilde\alpha} C_\alpha \otimes
A_{\widetilde\alpha}^* \right)^{-1}= \sum_{k=0}^\infty
\left(\sum_{|\alpha|\geq 1} a_{\widetilde\alpha} C_\alpha \otimes
A_{\widetilde\alpha}^* \right)^k
\end{equation*} is well-defined,
where the convergence is in the operator norm topology. Consequently,
using the definition of $\Lambda_1,\ldots, \Lambda_n$ and   relation \eqref{Po-ker}, we obtain
\begin{equation*}
\begin{split}
 K_{f,T}^{(m)}h=
(I_{F^2(H_n)} \otimes R^{1/2})\left( I-\sum_{|\alpha|\geq 1}
a_{\widetilde\alpha} \Lambda_\alpha \otimes  T_{\widetilde\alpha}^*
\right)^{-m}(1\otimes h), \qquad h\in \cH.
\end{split}
\end{equation*}
 Combining the above-mentioned results with the fact that
$K_q:=(P_{\cN_{\cP}}\otimes I_{\overline{
R^{1/2}(\cH)}})K^{(m)}_{f,A,R}$, one can complete the proof of our
assertion.

We remark that in the particular case when: $n=m=1$, $f=X$,
     $\cH=\CC$,
  $A=\lambda\in \DD$, $R=I$, and $\cP=\{0\}$,  we recover the Berezin transform \cite{Be} of a bounded
 operator on the Hardy space $H^2(\DD)$, i.e.,
$$
{\bf B}_\lambda [g]=(1-|\lambda|^2)\left<g k_\lambda,
k_\lambda\right>,\quad g\in B(H^2(\DD)),
$$
where $k_\lambda(z):=(1-\overline{\lambda} z)^{-1}$ and  $z,
\lambda\in \DD$.

\smallskip

The following technical lemma is a slight extension of Lemma 1.4 and
2.2 from  \cite{Po-Berezin}, where the operator $D$ was positive. In
our extension, $D$ is a self-adjoint operator and   the condition
(a) is new. However,  since the proof is similar  to those from
\cite{Po-Berezin}, we shall omit it. A linear map $\varphi:B(\cH)\to
B(\cH)$ is called power bounded if there exists a constant $M>0$
such that $\|\varphi^k\|\leq M$ for any $k\in \NN$.

\begin{lemma}\label{ineq-lim}
Let $\varphi:B(\cH)\to B(\cH)$ be a positive  linear map and  let $
D\in B(\cH)$ be a self-adjoint operator and $m\in \NN$. Then the
following statements hold:
\begin{enumerate}
\item[(i)] If $\varphi$ is power bounded, then
$$(id-\varphi)^m(D)\geq 0\quad \text{ if and only if }\quad  (id-\varphi)^s(D)\geq
0, \quad s=1,2,\ldots,m.
$$
 \item[ii)] Under
  either one of the  conditions:
\begin{enumerate}
\item[(a)]
$(id-\varphi)^s(D)\geq 0$ for  any $s=1,\ldots, m$, or
\item[(b)] $\varphi$ is  power bounded and $(id-\varphi)^m(D)\geq 0$,
\end{enumerate}
the following limit exists and
$$
\lim_{k\to\infty} k^d\left< \varphi^k(id-\varphi)^d(D)h,h\right>=
\begin{cases} \lim\limits_{k\to\infty}\left< \varphi^k (D)h,h\right> &
\quad \text{ if } \ d=0\\
0 & \quad \text{ if } \ d=1,2,\ldots, m-1
\end{cases}
$$
  for  any \ $h\in \cH$.
  \end{enumerate}
\end{lemma}

In what follows we also need the following result.
    For  information  on completely bounded (resp. positive) maps,  we refer
 to \cite{Pa-book} and  \cite{Pi}.

\begin{lemma}\label{wot-cp}  Let $f:=\sum_{|\alpha|\geq 1} a_\alpha X_\alpha$ be
a positive regular free holomorphic function
 and let $A:=(A_1,\ldots, A_n)\in B(\cH)^n$ be an $n$-tuple of
 operators such that
$\sum_{|\alpha|\geq 1} a_\alpha A_\alpha A_\alpha^* $ is convergent
in the weak operator topology.
 Then the map
$\Phi_{f,A}:B(\cH)\to B(\cH)$, defined by
$$\Phi_{f,A}(X)=\sum_{|\alpha|\geq 1} a_\alpha A_\alpha
XA_\alpha^*,\qquad X\in B(\cH),
$$
where the convergence is in the weak operator topology, is a
 completely positive linear map which is WOT-continuous on bounded
 sets. Moreover, if $0<r<1$, then
$$\Phi_{f,A}(X)=\text{\rm WOT-}\lim_{r\to1} \Phi_{f,rA}(X), \qquad
X\in B(\cH).
$$
\end{lemma}
\begin{proof}
Note that, for any $x,y\in \cH$ and any finite subset
$\Lambda\subset \{\alpha\in \FF_n^+:\ |\alpha|\geq 1\}$, we have
\begin{equation*}
\begin{split}
\sum_{\alpha\in \Lambda}\left| \left<a_\alpha A_\alpha
XA_\alpha^*x,y\right>\right| &\leq \|X\|\sum\limits_{\alpha\in
\Lambda} a_\alpha
\|A_\alpha^*x\|\|A_\alpha^*y\|\\
&\leq \|X\|\left(\sum\limits_{\alpha\in \Lambda} a_\alpha
\|A_\alpha^*x\|^2\right)^{1/2}\left(\sum\limits_{\alpha\in \Lambda}
a_\alpha \|A_\alpha^*y\|^2\right)^{1/2}.
\end{split}
\end{equation*}
Now, since $\sum_{|\alpha|\geq 1} a_\alpha A_\alpha A_\alpha^* $ is
convergent in the weak operator topology it is easy to see that the
series $\Phi_{f,A}(X)=\sum_{|\alpha|\geq 1} a_\alpha A_\alpha
XA_\alpha^*$ convergence is in the weak operator topology. Moreover,
the above-mentioned  inequality  is true for any subset $\Lambda $
in $ \{\alpha\in \FF_n^+:\ |\alpha|\geq 1\}$. In particular, we
deduce that
$$\left|\left<\Phi_{f,A}(X)x,y\right>\right|\leq
\|X\|\left<\Phi_{f,A}(I)x,x\right>^{1/2}
\left<\Phi_{f,A}(I)y,y\right>^{1/2},\qquad x,y\in \cH.
$$
On the other hand, since the map $\Phi^{(k)}_{f,A}(X):=\sum_{1\leq
|\alpha|\leq k} a_\alpha A_\alpha XA_\alpha^*$, $X\in B(\cH)$, is
completely positive  for each $k\in \NN$ and
$\Phi_{f,A}(X)=\text{\rm WOT-}\lim_{k\to\infty}\Phi^{(k)}_{f,A}(X)$, we
deduce that $\Phi_{f,A}$ is a completely positive map on $B(\cH)$.
Since $\sum\limits_{|\alpha|\geq 1} a_\alpha A_\alpha A_\alpha^* $
is convergent in the weak operator topology, for any $\epsilon>0$
and $x,y\in \cH$, there is $N_0\in \NN$ such that
$$
\sum_{|\alpha|>N_0}\left<a_\alpha A_\alpha
A_\alpha^*x,x\right><\epsilon \quad \text{ and } \quad
\sum_{|\alpha|>N_0}\left<a_\alpha A_\alpha
A_\alpha^*y,y\right><\epsilon.
$$
Using the above-mentioned inequalities, we deduce that
$$
\sum_{|\alpha|>N_0}\left|\left<a_\alpha A_\alpha
XA_\alpha^*x,y\right>\right|\leq \epsilon \|X\|.
$$
Now, it is easy to see that $\Phi_{f, A}$ is WOT-continuous on bounded
 sets. On the other hand, we also have $
\sum_{|\alpha|>N_0}\left|\left<a_\alpha r^{|\alpha|}A_\alpha
XA_\alpha^*x,y\right>\right|\leq \epsilon \|X\|$ for any $r\in
[0,1]$. This can be used to show that $\Phi_{f,A}(X)=\text{\rm
WOT-}\lim_{r\to1} \Phi_{f,rA}(X)$ for any $ X\in B(\cH).$ The proof
is complete.
\end{proof}
We remark that Lemma \ref{wot-cp} remains true if
$\{a_\alpha\}_{|\alpha|\geq 1}$  is just  a sequence of positive
numbers
 and $A:=(A_1,\ldots, A_n)\in B(\cH)^n$ is  an $n$-tuple of
 operators such that
$\sum_{|\alpha|\geq 1} a_\alpha A_\alpha A_\alpha^* $ is convergent
in the weak operator topology.

\bigskip

We denote by  $C (f,A)^+$ the cone of all positive operators $D\in B(\cH)$ such that
$$(id-\Phi_{f, A})^s(D)\geq 0\quad \text{ for  } \  s=1,\ldots,
m.
$$
 We denote by  $ C_{rad} (f,A)^+$   the set of all    operators
 $D\in C  (f,A)^+$  such that  there is $\delta\in (0,1)$ with the
 property that
$D\in C (f,rA)^+$ for any   $r\in (\delta, 1]$.

A few examples are necessary. Note that if $m=1$ then  we always
have $C (f,A)^+=C_{rad} (f,A)^+$.  We  remark  that if $m\geq 2$ and
$p=a_1X_1+\cdots + a_nX_n$, $a_i>0$, then we also have $C
(p,A)^+=C_{rad} (p,A)^+$. Indeed, it is enough to see that if
   $0<r\leq 1$,  then
\begin{equation*}
\begin{split}
(id-\Phi_{p,rA})^k(D)&=\left[(id-\Phi_{p,A})+(1-r)\Phi_{p,A}\right]^k(D)\\
&= \sum_{j=0}^k \left(\begin{matrix} k\\j
\end{matrix}\right) (1-r)^{k-j}\Phi_{p,A}^{k-j}(id-\Phi_{p,A})^j(D)
\end{split}
\end{equation*}
for any $k=1,\ldots,m$. Since $(id-\Phi_{p,A})^j(D)\geq 0$ for
$j=1,\ldots,m$ and using the fact that  $\Phi_{p,A}^j$ is a positive
linear map, we deduce that $(id-\Phi_{p,rA})^k(D)\geq 0$ for
$k=1,\ldots,m$ and $r\in (0,1]$, which proves our assertion. Note
also that when $m\geq 1$ and   $q$ is any positive regular
noncommutative polynomial so that, for  each  $s=1,\ldots, m$,
$(id-\Phi_{q, A})^s(D)$ is a positive invertible operator, then
$D\in C_{rad} (q,A)^+$.

We say that
  ${\bf D}_f^m (\cH)$  is a {\it radial domain} if  there exists  $\delta\in (0,1)$ such that
$(rW_1,\ldots, rW_n)\in {\bf D}_f^m (F^2(H_n))$ for any
    $r\in (\delta ,1]$, where $(W_1,\ldots, W_n)$ is the universal model associated with ${\bf
    D}_f^m$. We remark  that the notion of radial domain does not
    depend on the Hilbert space $\cH$.
      Note that if $m=1$,
then ${\bf D}_f^1(\cH)$ is always a   radial  domain. This case was
extensively studied in \cite{Po-domains}. When $m\geq 2$, we  point
out    the particular case    $p:=a_1X_1+\cdots
+a_n X_n$, $a_i>0$, when  ${\bf D}_p^m(\cH)$  is also a radial domain.

 Now, we are ready to   show that, for radial domains ${\bf D}_f^m (\cH)$,
  the elements of the noncommutative cone  $ C_{rad} (f,A)^+$
are in one-to-one correspondence
 with the elements of a class
of noncommutative Berezin transforms.
\begin{theorem}\label{poisson} Let  ${\bf D}_f^m (\cH)$ be a radial domain,
where  $f:=\sum_{ |\alpha|\geq 1} a_\alpha X_\alpha$
 is    a positive regular  free holomorphic function and $m\geq 1$.
Let $\cP$ be a family of
noncommutative homogeneous
 polynomials  and let  $B:=(B_1,\ldots, B_n)$ be
the universal
  model associated with the noncommutative variety
 $\cV^m_{f,\cP}$. If $A:=(A_1,\ldots,
A_n)\in B(\cH)^n$ is such that $\sum_{|\alpha|\geq 1} a_\alpha A_\alpha A_\alpha^* $ \ is
SOT-convergent and
$p(A_1,\ldots, A_n)=0$, \  $p\in \cP,
$
then there is a bijection
$$\Gamma: CP(A,\cV^m_{f,\cP})\to C_{rad}(f,A)^+, \qquad
 \Gamma(\varphi):= \varphi (I),
 $$
  where  $  CP(A,\cV^m_{f,\cP})$ is   the set
  of all completely positive linear maps \ $\varphi : \cS_{f,\cP}
 \to B(\cH)$ such that
$$ \varphi(B_\alpha B_\beta^*) =A_\alpha \varphi(I) A_\beta^*,
 \qquad \alpha, \beta \in \FF_n^+,
 $$
 where $\cS_{f,\cP}:=
\overline{\text{\rm span}} \{ B_\alpha B_\beta^*:\ \alpha,\beta\in
\FF_n^+\}$. Moreover,  if $D\in C_{rad}(f,A)^+$, then
$\Gamma^{-1}(D)$ coincides
 with the noncommutative
Berezin  transform  associated with $q:=(f, m, A,R,\cP)$ and defined by
$$
  \overline{{\bf B}}_{q}[\chi]:=\lim_{r \to 1}
 K_{q_r }^* (\chi\otimes I) K_{q_r} ,\qquad \chi\in
   \cS_{f,\cP},
$$
where  $q_r:=(f, m, rA,R_r,\cP)$ and $R_r:=(id-\Phi_{f,rA})^m(D)$,
$r\in [0,1]$, and the limit exists in the operator norm topology.
\end{theorem}
\begin{proof} We recall  that the subspace $\cN_\cP\neq \{0\}$ is
invariant under each operator $W_1^*,\ldots, W_n^*$ and
$B_i:=P_{\cN_\cP}W_i|_{\cN_\cP}$, $i=1,\ldots, n$. Setting
$B:=(B_1,\ldots, B_n)$ and taking into account that
$\Phi_{f,W}(I)\leq I$,  we deduce that
 $\Phi_{f,B}(I)\leq I$ and, consequently,  $\Phi_{f,rB}(I)=\sum_{k=1}^\infty
\sum_{|\alpha|=k} a_\alpha r^{|\alpha|} B_\alpha B_\alpha ^*\leq
 I$, where the convergence is in the operator norm topology. This implies  $\Phi_{f,rB}(I)\in  \cS_{f,\cP}$  for any $r\in [0,1)$.
The fact that  ${\bf
 D}_{f}^m$ is a radial domain  implies   $(rW_1,\ldots, rW_n)\in {\bf D}_{f
 }^m(F^2(H_n))$, $r\in (\delta, 1)$, for some $\delta\in (0,1)$  and, consequently,    $(id-\Phi_{f,rB})^s(I)\geq 0$ for $s=1,\ldots,m$ and $r\in (\delta,1)$.
Since
$$
\Phi_{f,rB}^j(I)=\sum_{k=1}^\infty \sum_{|\alpha|=k} a_\alpha
r^{|\alpha|} B_\alpha  \Phi_{f,rB}^{j-1}(I)B_\alpha^*,\qquad j\in \NN,
$$
 and
$\|\Phi_{f,rB}^k(I)\|\leq 1$ for any $k\in \NN$, it is clear that
$\Phi_{f,rB}^j(I)\in \cS_{f,\cP}$. Taking into account that
$$
  (id-\Phi_{f,r  B})^s(I)=\sum_{j=0}^s (-1)^j
\left(\begin{matrix} s\\j\end{matrix}\right) \Phi_{f,r B}^j(I), \qquad j\in \NN,
$$
 we deduce that  $(id-\Phi_{f,r
B})^s(I)\in \cS_{f,\cP}$ for $s=1,\ldots,m$. Now, assume that
$\varphi : \cS_{f,\cP}\to B(\cH) $  is a completely positive linear  map such
that
$$ \varphi(B_\alpha B_\beta^*) =A_\alpha \varphi(I) A_\beta^*,
 \qquad \alpha, \beta \in \FF_n^+.
 $$
Then, setting $D:=\varphi(I)$, we deduce that $D\geq 0$ and
$$
(id-\Phi_{f,rA})^s(D)=\varphi \left(
(id-\Phi_{f,rB})^s(I)\right)\geq 0,\qquad r\in (\delta,1),
$$
for any $ s=1,\ldots,m.$ Since  the series
$\sum\limits_{|\alpha|\geq 1} a_\alpha A_\alpha A_\alpha^* $ \ is
SOT-convergent  one can  use Lemma \ref{wot-cp} to deduce  that
$\Phi^k_{f,A}(D)=\text{\rm WOT-}\lim_{r\to1}\Phi_{f,rA}^k(D)$ for
$k\in\NN$ and, moreover,
$$(id-\Phi_{f,A})^s(D)=\text{\rm
WOT-}\lim_{r\to1}(id-\Phi_{f,rA})^s(D)\geq 0
$$
for any $s=1,\ldots,m$. This shows that $D\in  C_{rad}(f,A)^+$. To
prove that $\Gamma$ is one-to-one, let $\varphi_1$ and $\varphi_2$
be completely positive linear maps on $\cS_{f,\cP}$ such that
$\varphi_j(B_\alpha B_\beta^*)=A_\alpha \varphi_i(I)A_\beta$,
$\alpha, \beta\in \FF_n^+$, and assume that
$\Gamma(\varphi_1)=\Gamma(\varphi_2)$, i.e.,
$\varphi_1(I)=\varphi_2(I)$. Then we have $\varphi_1(B_\alpha
B_\beta^*)=\varphi_2(B_\alpha B_\beta^*)$ for $\alpha, \beta\in
\FF_n^+$. Taking into account  the continuity of $\varphi_1$ and
$\varphi_2$ in the operator norm, we deduce that
$\varphi_1=\varphi_2$.

 To prove surjectivity, fix
$D\in C_{rad}(f,A)^+$. Then
 $D\in  B(\cH)$  is a positive operator with the property
  that  there is $\delta\in (0,1)$ such that
 $(id-\Phi_{f,rA})^s(D)\geq 0$ for any  $s=1,\ldots, m$ and $r\in (\delta,
 1)$. Since the set $\cP$ consists of homogeneous noncommutative polynomials , we have
 $p(rA_1,\ldots, rA_n)=0$ for any $p\in \cP$ and $r\in (\delta, 1)$. We show now that, for each $r\in (\delta, 1)$,
 the tuple $q_r:=(f,m,rA,R_r,\cP)$, where $R_r:=(id-\Phi_{f,rA})^m(D)$,  is compatible. Indeed, we can use
the equality
$$\left(\begin{matrix}
i+j\\ j
\end{matrix}\right)-\left(\begin{matrix}
i+j-1\\ j
\end{matrix}\right)=\left(\begin{matrix}
i+j-1\\ j-1
\end{matrix}\right),\qquad i,j\in \NN,
$$
and Lemma \ref{ineq-lim} to  obtain
\begin{equation*}\begin{split}
\sum_{k=0}^\infty \left(\begin{matrix} k+m-1\\ m-1
\end{matrix}\right) \Phi_{f,rA}^k(R_r )&=
D- \text{\rm WOT-}\lim_{k\to \infty}\sum_{j=0}^{m-1}\left(\begin{matrix} k+j\\ j
\end{matrix}\right)\Phi_{f,rA}^{k+1}(id-\Phi_{f,rA})^j(D)\\
&= D-\text{\rm WOT-}\lim_{k\to \infty}  \Phi_{f,rA}^k(D).
\end{split}
\end{equation*}
Since $\Phi_{f,rA}^k(D)\leq r^{2k}\Phi_{f,A}^k(D)\leq  r^{2k} D$, we
have $\text{\rm WOT-}\lim_{k\to\infty}  \Phi_{f,rA}^k(D)=0$. Therefore, we deduce that
\begin{equation}
\label{Kr*}
\sum_{k=0}^\infty  \left(\begin{matrix} k+m-1\\ m-1
\end{matrix}\right)\Phi_{f,rA}^k(R_r )=D, \qquad r\in (\delta, 1).
\end{equation}
According  to Lemma \ref{lemma2},  the constrained noncommutative Berezin kernel  \ $K_{q_r}$,
$r\in (\delta,1)$,
associated with the compatible tuple $q_r:=(f,rA,R_r,\cP)$,  has the property that
\begin{equation} \label{KK}
  K_{q_r} (r A_i^*)=(B_i^*\otimes
I_\cH)K_{q_r},\qquad i=1,\ldots, n,
\end{equation}
where $(B_1,\ldots, B_n)$ is the $n$-tuple of constrained weighted
left creation operators associated with the noncommutative variety
$\cV^m_{f,\cP}$, and
\begin{equation*}
  K_{q_r}^* K_{q_r}= \sum_{k=0}^\infty \left(\begin{matrix} k+m-1\\ m-1
\end{matrix}\right)
\Phi_{f,rA}^k(R_r),\qquad r\in (\delta, 1),
\end{equation*}
where $R_r:=(id-\Phi_{f,rA})^m(D)$.
Hence and using relation \eqref{Kr*},  we obtain
\begin{equation}\label{K*K}
K_{q_r}^* K_{q_r}= D, \qquad r\in (\delta, 1).
\end{equation}
For each  $r\in (\delta, 1)$, define the operator ${\bf B}_{q_r} : \cS_{f,\cP}  \to B(\cH)$
 by setting
 \begin{equation}\label{pois}
 {\bf B}_{q_r}(\chi):= K_{q_r}^* (\chi\otimes I_\cH) K_{q_r},
 \qquad \chi\in \cS_{f,\cP}.
 \end{equation}
  Using relation \eqref{KK} and \eqref{K*K},   we have
\begin{equation}\label{KSK}
K_{q_r}^* (B_\alpha B_\beta^*\otimes I) K_{q_r}= r^{|\alpha|+ |\beta|} A_\alpha DA_\beta^*,\qquad
\alpha,\beta\in \FF_n^+, \ r\in (\delta, 1).
\end{equation}
Hence,  and using relations \eqref{K*K} and  \eqref{pois}, we infer
that ${\bf B}_{q_r}$   is a completely positive  linear map with
${\bf B}_{q_r}(I)=D$ and  $\|{\bf B}_{q_r}\|=\|D\|$ for
$r\in (\delta, 1)$.

Now, we show that $\lim_{r \to 1} {\bf B}_{q_r}(\chi)$ exists in the operator norm topology for each $\chi\in \cS_{f,\cP}$.
Given  a polynomial $\varphi(B_1,\ldots, B_n):= \sum_{\alpha,\beta\in \FF_n^+} a_{\alpha
\beta}B_\alpha B_\beta^*$  in the operator system $ \cS_{f,\cP}$, we define
\begin{equation*}
\varphi_D(A_1,\ldots, A_n):=
\sum_{\alpha,\beta\in \FF_n^+} a_{\alpha \beta}A_\alpha DA_\beta^*.
\end{equation*}
The definition is correct since, according to  relation
\eqref{KSK}, we have the following von Neumman type inequality
\begin{equation}\label{von2}
\|\varphi_D(A_1,\ldots, A_n)\|\leq \|D\|\|\varphi(B_1,\ldots, B_n)\|.
\end{equation}
Now, fix $\chi\in   \cS_{f,\cP}$ and let
 $\varphi^{(k)}(B_1,\ldots, B_n)$ be a sequence of polynomials in
 $ \cS_{f,\cP}$ convergent to $\chi$, in the operator norm topology.
 Define the operator
 \begin{equation}\label{fd}
\chi_D(A_1,\ldots, A_n):= \lim_{k\to\infty}\varphi^{(k)}_D(A_1,\ldots, A_n).
\end{equation}
 Taking into account relation \eqref{von2}, it is clear that the
  operator $\chi_D(A_1,\ldots, A_n)$ is well-defined and
 \begin{equation*}
\|\chi_D(A_1,\ldots, A_n)\|\leq \|D\|\|\chi\|.
\end{equation*}
According to relation  \eqref{KSK}, we have
$$
\|\varphi^{(k)}_D(rA_1,\ldots, rA_n)\|\leq \|D\|\|\varphi^{(k)}(B_1,\ldots, B_n)\|,
$$
for any $r\in (\delta, 1)$.
  Taking into account that ${\bf B}_{q_r}$ is a bounded linear operator and using again
   relation \eqref{KSK}, we deduce that
 \begin{equation}\label{fdr}\begin{split}
  \lim_{k\to\infty} \varphi^{(k)}_D(rA_1,\ldots, rA_n)
 =\lim_{k\to\infty}{K}_{q_r}^*(\varphi^{(k)}(B_1,\ldots, B_n)\otimes I)K_{q_r}
 ={\bf B}_{q_r}[\chi],
 \end{split}
 \end{equation}
  for any $r\in (\delta, 1)$.
 Using  relations \eqref{fd}, \eqref{fdr}, the fact that
  $\|\chi-\varphi^{(k)}(B_1,\ldots, B_n)\|\to 0$ as $k\to \infty$, and
 $$
 \lim_{r\to 1}\varphi^{(k)}_D(rA_1,\ldots, rA_n)= \varphi^{(k)}_D(A_1,\ldots, A_n),
 $$
  we can deduce  that
 $$
 \lim_{r \to 1} {\bf B}_{q_r}[\chi]=\chi_D(A_1,\ldots, A_n)
 $$
 in the norm topology. Indeed, note that
 \begin{equation*}
 \begin{split}
&\|\chi_D(A_1,\ldots, A_n)-{\bf B}_{q_r}[\chi]\|\\
&\leq \|\chi_D(A_1,\ldots, A_n)-\varphi^{(k)}_D(A_1,\ldots,
A_n)\|+\|\varphi^{(k)}_D(A_1,\ldots, A_n)- {\bf B}_{q_r}(\varphi^{(k)})\|\\
&\quad + \|{\bf B}_{q_r}(\varphi^{(k)})-{\bf B}_{q_r}(\chi)\|\\
&\leq \|\chi-\varphi^{(k)}(B_1,\ldots, B_n)\|\|D\|+\|\varphi^{(k)}_D(A_1,\ldots,
A_n)-\varphi^{(k)}_D(rA_1,\ldots, rA_n)\|\\
&\quad +\|\chi-\varphi^{(k)}(B_1,\ldots, B_n)\|\|D\|.
 \end{split}
 \end{equation*}
 For any $r\in (\delta, 1)$,  ${\bf B}_{q_r}$ is a completely positive linear  map.
  Hence, and using
  relation \eqref{KSK},
    we infer  that
 $$
  \overline{{\bf B}}_{q}[\chi]:=\lim_{r \to 1}
 K_{q_r }^* (\chi\otimes I) K_{q_r}, \qquad   \chi\in
   \cS_{f,\cP},
$$
 is a completely positive map with $  \overline{{\bf B}}_q (I)=D$ and
 $\overline{{\bf B}}_q (B_\alpha B_\beta^*)=A_\alpha \overline{{\bf B}}_q(I)A_\beta$,
$\alpha, \beta\in \FF_n^+$.
 The proof is
complete.
\end{proof}
The following result is an extension of the noncommutative von
Neumann inequality (see \cite{von}, \cite{Po-von},
\cite{Po-poisson}, \cite{Po-Berezin}).

\begin{corollary}\label{VN} Under  the hypotheses of Theorem \ref{poisson}, if
$D\in C_{rad}(f,A)^+$, then
  we have the following von
Neumann type inequality:
\begin{equation*}
  \left\|\sum_{\alpha,\beta\in \Lambda} A_\alpha D
A_\beta^*\otimes C_{\alpha,\beta}\right\| \leq \|D\|
\left\|\sum_{\alpha,\beta\in \Lambda}
 B_\alpha B_\beta^*\otimes C_{\alpha,\beta}\right\|
 \end{equation*}
 for any
  finite  set $\Lambda\subset \FF_n^+$ and $C_{\alpha,\beta}\in
  B(\cE)$, where $\cE$ is a Hilbert space. If, in addition, $D$ is
  an invertible operator, then  the map   $u:
  \cA_n(\cV^m_{f,\cP})\to B(\cH)$ defined by
  $$
  u(p(B_1,\ldots, B_n)):=p(A_1,\ldots, A_n)
  $$
  is completely bounded with $\|u\|_{cb}\leq \|D^{-1/2}\|\|D^{1/2}\|$.
\end{corollary}
\begin{proof} Due to relation \eqref{KSK}, we have
\begin{equation*}
(K_{q_r}^* \otimes I_\cE)(B_\alpha B_\beta^*\otimes I\otimes
C_{\alpha,\beta}) (K_{q_r}\otimes I_\cE)= r^{|\alpha|+ |\beta|}
A_\alpha DA_\beta^*\otimes C_{\alpha,\beta},\qquad \alpha,\beta\in
\FF_n^+, \ r\in (\delta, 1).
\end{equation*}
Since $K_{q_r}^* K_{q_r}= D$ for $ r\in (\delta, 1)$, one can easily
deduce the von Neumann type inequality. To prove the second part,
note that, if $D$ is invertible, then the first part of this
corollary implies
\begin{equation*}
\begin{split}
\|p(A_1,\ldots, A_n)\|^2&\leq \|D^{-1/2}\|^2\|p(A_1,\ldots,
A_n)D^{1/2}\|^2\\
&= \|D^{-1/2}\|^2 \|p(A_1,\ldots, A_n)Dp(A_1,\ldots, A_n)^*\|\\
&\leq \|D^{-1/2}\|^2 \|D\| \|p(B_1,\ldots, B_n)p(B_1,\ldots, B_n)^*\|\\
&=\|D^{-1/2}\|^2\|D^{1/2}\|^2 \|p(B_1,\ldots, B_n)\|^2
\end{split}
\end{equation*}
for any noncommutative polynomial $p$. A similar result holds if we
pass to matrices. Therefore, we  deduce that $u$ is completely
bounded with $\|u\|_{cb}\leq \|D^{-1/2}\|\|D^{1/2}\|$. The proof is
complete.
\end{proof}

 \begin{example}
\begin{enumerate}
\item[(i)] When $m=1$,
$f=X_1+\cdots + X_n$, and $D=I$, we obtain the noncommutative Poisson
transform introduced in \cite{Po-poisson} (case  $\cP=\{0\}$) and \cite{Po-varieties}
(case $\cP\neq \{0\}$).
\item[(ii)] When $m=1$,
$f=X_1+\cdots + X_n$, $\cP=\{0\}$,  and $D\geq 0$ such that
$\sum_{i=1}^n A_i DA_i^*\leq D$, we obtain the noncommutative
Poisson transform from \cite{Po-similarity}.
\item[(iii)] When $m \geq 1$, $D=I$, and  $f$ is an arbitrary positive regular free holomorphic
function,   we obtain the noncommutative  Berezin transforms
associated with noncommutative domains ${\bf D}_f^m$ or
noncommutative varieties $\cV^m_{f,\cP}$, which were studied in \cite{Po-Berezin} and
\cite{Po-domains}.
\end{enumerate}
\end{example}

\bigskip

\section{ Generalized noncommutative Berezin transforms and the cone $C_{pure}(f,A)^+$}

In this section, we  study the noncommutative cone $C_{pure}(f,A)^+$
of all  pure solutions of the operator inequalities $
(id-\Phi_{f,A})^s(X)\geq 0$, $ s=1,\ldots,m.$ When $A$ is  a pure
$n$-tuple of operators in  the noncommutative variety
$\cV^1_{f,\cP}(\cH)$, we obtain a complete description  of the
noncommutative cone $C(f,A)^+$.

Let $A:=(A_1,\ldots, A_n)\in B(\cH)^n$ be such that
$\sum_{|\alpha|\geq 1} a_\alpha A_\alpha A_\alpha^* $ is
convergent in the weak  operator topology and recall that
$$
 \Phi_{f,A}(X):=\sum\limits_{|\alpha|\geq 1} a_\alpha A_\alpha
XA_\alpha^*,\qquad X\in B(\cH),
$$
where the convergence is in the weak operator topology. We assume
that $\Phi_{f,A}$ is power bounded. A self-adjoint operator $C\in
B(\cH)$ is called pure solution of the inequality
$(id-\Phi_{f,A})^m(X)\geq 0$ if
$$(id-\Phi_{f,A})^m(C)\geq 0\quad  \text{\rm and } \quad
\text{SOT-}\lim_{k\to \infty}  \Phi_{f,A}^k (C)=0.
$$
 Note that since $\Phi_{f,A}$ is power bounded,  Lemma \ref{ineq-lim} implies
  $\Phi_{f,A}(C)\leq C$.   This can be used to show that a pure  self-adjoint solution is always
a positive operator. In what follows we present a {\it canonical
decomposition}
 for the self-adjoint solutions of the operator inequality $(id-\Phi_{f,A})^m(X)\geq 0$.

\begin{theorem}\label{decomp}  Let
  $f:=\sum_{|\alpha|\geq 1} a_\alpha X_\alpha$
   be
a positive regular free holomorphic function  and $m\geq 1$. Let
$A:=(A_1,\ldots, A_n)\in B(\cH)^n$ be such that
$\sum\limits_{|\alpha|\geq 1} a_\alpha A_\alpha A_\alpha^* $ is
convergent in the weak operator topology and  $\Phi_{f,A}$ is power
bounded. If  $Y=Y^*\in B(\cH)$ is   such that
    $(id-\Phi_{f,A})^m(Y)\geq 0$,
  then there exist operators $B,C\in B(\cH)$ with the following properties:
  \begin{enumerate}
  \item[(i)] $Y=B+C$;
  \item[(ii)] $B=B^*$  and $  \Phi_{f,A}(B)=B$;
  \item[(iii)] $C\geq 0$,   $(id-\Phi_{f,A})^m(C)\geq 0$, and \  $\text{\rm
SOT-}\lim_{k\to \infty}  \Phi_{f,A}^k (C)=0$.

  \end{enumerate}
   Moreover, the decomposition $Y=B+C$ is unique with the above-mentioned properties and
   $$
B= \text{\rm SOT-}\lim\limits_{k\to \infty}\Phi_{f,A}^k(Y)=\text{\rm
SOT-}\lim_{k\to\infty} \frac{1}{k}\sum_{j=0}^{k-1} \Phi_{f,A}^j(Y).
 $$
\end{theorem}

\begin{proof} Let $Y=Y^*\in B(\cH)$ be    such that
    $(id-\Phi_{f,A})^m(Y)\geq 0$. Since $\Phi_{f,A}$ is power
    bounded, Lemma \ref{ineq-lim} implies $\Phi_{f,A}(Y)\leq Y$.
Consequently, the sequence  of self-adjoint operators
$\{\Phi_{f,A}^k(Y)\}_{k=0}^\infty$
 is bounded and decreasing. Thus it converges strongly to  a selfadjoint
  operator $B:= \text{\rm SOT-}\lim\limits_{k\to \infty}\Phi_{f,A}^k(Y)$.
  Since $\Phi_{f,A}$ is a  $WOT$-continuous map,
  we  have $\Phi_{f,A}(B)=B$. Note that  $C:=Y-B\geq 0$ satisfies
  the inequality
  $
  \Phi_{f,A}(C) \leq C$, and $(id-\Phi_{f,A})^m(C)=(id-\Phi_{f,A})^m(Y)\geq
  0.$ Moreover, $\Phi_{f,A}^k(C)\to 0$ strongly, as $k\to \infty$.

 To prove the uniqueness of the decomposition,  suppose $Y=B_1+C_1$,
 where $B_1$ and $C_1$  have the same properties as $B$ and $C$,
 respectively.
 Then
 $$
 B-B_1= \Phi_{f,A}^k(B-B_1)=\Phi_{f,A}^k(C_1-C),  \qquad k\in \NN.
 $$
 Taking  $k\to \infty$, we get   $B=B_1$ and, consequently, $C=C_1$.
Since   $0\leq\Phi_{f,A}^k(C)\leq C$, $k\in \NN$,  and  $\text{\rm
SOT}-\lim\limits_{k\to\infty}
 \Phi_{f,A}^k(C)=0$,
 a standard argument shows that
 $
\text{\rm SOT-}\lim_{k\to\infty} \frac{1}{k}\sum_{j=0}^{k-1}
\Phi_{f,A}^j(C) =0.
 $
 On the other hand,
 since $Y=B+C$ and $\Phi_{f,A}(B)=B$, we infer that
 $$
 \frac{1}{k}\sum_{j=0}^{k-1}
\Phi_{f,A}^j(Y)=B+\frac{1}{k}\sum_{j=0}^{k-1} \Phi_{f,A}^j(C).
 $$
 Hence, the result follows. The proof is complete.
\end{proof}

We denote by $C_{pure}(f,A)^+$   the set of all  operators $D\in
B(\cH)$ such that
$$(id-\Phi_{f,A})^s(D)\geq 0, \qquad s=1,\ldots,m,
$$
 and $\Phi_{f,A}^k(D)\to 0$ strongly, as $k\to\infty$. Note that
such an operator $D$ is always positive.

\begin{theorem}\label{poisson2} Let  $f:=\sum_{|\alpha|\geq 1} a_\alpha X_\alpha$
 be   a positive regular  free holomorphic function and $m\geq 1$.
Let $\cP$ be a family of noncommutative
 polynomials  with $\cN_\cP\neq \{0\}$ and let  $B:=(B_1,\ldots, B_n)$ be
the universal
  model associated with the noncommutative variety
 $\cV^m_{f,\cP}$. If $A:=(A_1,\ldots,
A_n)\in B(\cH)^n$ is such that $\sum_{|\alpha|\geq 1} a_\alpha
A_\alpha A_\alpha^* $ \ is SOT-convergent and
$p(A_1,\ldots, A_n)=0$, \ $ p\in \cP$,
then there is a bijection
$$\Gamma: CP^{w^*}(A,\cV^m_{f,\cP})\to C_{pure}(f,A)^+, \qquad
 \Gamma(\varphi):= \varphi (1),
 $$
  where $CP^{w^*}(A,\cV^m_{f,\cP})$ is   the set
  of all $w^*$-continuous  completely positive linear maps \ $\varphi : \cS_{f,\cP}^{w^*}
 \to B(\cH)$ such that
$$ \varphi(B_\alpha B_\beta^*) =A_\alpha \varphi(I) A_\beta^*,
 \qquad \alpha, \beta \in \FF_n^+,
 $$
 where $\cS_{f,\cP}^{w^*}:=
\overline{\text{\rm span}}^{w^*} \{ B_\alpha B_\beta^*:\
\alpha,\beta\in \FF_n^+\}$.  In addition,  if $D\in
C_{pure}(f,A)^+$, then $\Gamma^{-1}(D)$ coincides
 with the noncommutative
Berezin  transform  associated with $q:=(f, m, A,R,\cP)$ and defined
by
$$
  {\bf B}_{q}[\chi]:=
 K_{q }^* (\chi\otimes I) K_{q} ,\qquad \chi\in
   \cS_{f,\cP}^{w^*},
$$
where   $R:=(id-\Phi_{f,A})^m(D)$.

Moreover, an operator $D\in B(\cH)$ is in $ C_{pure}(f,A)^+$
 if and only if there is a Hilbert space $\cD$ and an operator
 $K:\cH\to \cN_\cP\otimes \cD$ such that
 \begin{equation*}
 D=K^*K\quad \text{ and } \quad KA_i^*= (B_i^*\otimes I_\cD)K, \qquad  i=1,\ldots, n.
 \end{equation*}
\end{theorem}

\begin{proof} Assume that
$\varphi : \cS_{f,\cP}^{w^*}
 \to B(\cH)$ is a $w^*$-continuous completely positive linear
map such that
$$ \varphi(B_\alpha B_\beta^*) =A_\alpha \varphi(I) A_\beta^*,
 \qquad \alpha, \beta \in \FF_n^+.
 $$
Then, setting $D:=\varphi(I)$ and taking into account that
$\Phi_{f,B}=\sum_{|\alpha|\geq 1} a_\alpha B_\alpha B_\alpha^*$ is
SOT convergent, we deduce that
$$
(id-\Phi_{f,A})^s(D)=\varphi \left( (id-\Phi_{f,B})^s(I)\right)\geq
0, \qquad  s=1,\ldots,m.
$$
On the other hand, recall that $\{\Phi_{f,B}^k(I)\}_{k=1}^\infty$ is
a
 bounded decreasing sequence of positive operators  which converges
 strongly to $0$, as $k\to \infty$. Since
 $\Phi_{f,A}^k(D)=\varphi(\Phi_{f,B}^k(I))$ for all  $k\in \NN$,
 one can easily see that
  $\{\Phi_{f,A}^k(D)\}_{k=1}^\infty$ is
a
 bounded decreasing sequence of positive operators  which converges
 strongly, as $k\to \infty$.
Taking into account that  $\varphi$ is
  continuous  in the  $w^*$-topology, which coincides with the weak
  operator topology   on
bounded sets, we deduce that $\Phi_{f,A}^k(D)\to 0$ strongly, as
$k\to \infty$. Therefore, $D\in C_{pure}(f,A)^+$.
 To
prove that $\Gamma$ is one-to-one, let $\varphi_1$ and $\varphi_2$
be $w^*$-continuous completely positive linear maps on
$\cS_{f,\cP}^{w^*}$ such that $\varphi_j(B_\alpha
B_\beta^*)=A_\alpha \varphi_j(I)A_\beta$, $\alpha, \beta\in
\FF_n^+$, and assume that $\Gamma(\varphi_1)=\Gamma(\varphi_2)$,
i.e., $\varphi_1(I)=\varphi_2(I)$. Then we have $\varphi_1(B_\alpha
B_\beta^*)=\varphi_2(B_\alpha B_\beta^*)$ for $\alpha, \beta\in
\FF_n^+$. Since     $\varphi_1$ and $\varphi_2$ are
$w^*$-continuous, we deduce that $\varphi_1=\varphi_2$.

We prove now that $\Gamma$ is a surjective map. Let  $D\in
C_{pure}(f,A)^+$ be fixed. According to Lemma \ref{lemma2}, the
constrained noncommutative Berezin kernel \ $K_{q}$  associated with
the compatible tuple $q:=(f,m,A,R,\cP)$, has the property that
\begin{equation}\label{KABK}
  K_{q}  A_i^*=(B_i^*\otimes
I_\cH)K_{q},\quad i=1,\ldots, n,
\end{equation}
where $(B_1,\ldots, B_n)$ is the  universal model associated with
the noncommutative variety $\cV^m_{f,\cP}$, and
\begin{equation*}
  K_{q}^* K_{q}= \sum_{k=0}^\infty \left(\begin{matrix} k+m-1\\ m-1
\end{matrix}\right)
\Phi_{f,A}^k(R),
\end{equation*}
where $R:=(id-\Phi_{f,A})^m(D)$.  As in the proof of Theorem
\ref{poisson}, we can use   Lemma \ref{ineq-lim}  and the fact that
$\text{\rm WOT-}\lim_{k\to\infty}  \Phi_{f,A}^k(D)=0$,  to obtain
\begin{equation*} K_q^* K_q=
\sum_{k=0}^\infty \left(\begin{matrix} k+m-1\\ m-1
\end{matrix}\right) \Phi_{f,A}^k(R )
= D-\text{\rm WOT-}\lim_{k\to \infty}  \Phi_{f,A}^k(D)=D
\end{equation*}
Define the operator ${\bf B}_{q} : \cS^{w^*}_{f,\cP}  \to B(\cH)$
 by setting
 \begin{equation*}
 {\bf B}_{q}(\chi):= K_{q}^* (\chi\otimes I_\cH) K_{q},
 \qquad \chi\in \cS^{w^*}_{f,\cP}.
 \end{equation*}
   Now,  due to  relation \eqref{KABK} it is easy to see that
\begin{equation*}
{\bf B}_{q}( B_\alpha B_\beta^*)=K_{q}^*(B_\alpha B_\beta^*\otimes
I) K_{q}= A_\alpha DA_\beta^*,\qquad \alpha,\beta\in \FF_n^+.
\end{equation*}
 Consequently,  ${\bf B}_{q}\in CP^{w^*}(A,\cV^m_{f,\cP})$    has
 the required properties.


To prove the last part of the theorem, note that the direct
implication follows if we take $K$ to be the noncommutative Berezin
kernel $K_q$.
   To prove the converse,  assume that there is a Hilbert space $\cD$ and an operator
 $K:\cH\to \cN_\cP\otimes \cD$ such that
 \begin{equation*}
 D=K^*K\quad \text{ and } \quad KA_i^*= (B_i^*\otimes I_\cD)K, \qquad  i=1,\ldots, n.
 \end{equation*}
 Then
$$
(id-\Phi_{f,A})^s(D)=K^* \left[ (id-\Phi_{f,B})^s(I)\otimes I_\cD
\right]K\geq 0, \qquad  s=1,\ldots,m.
$$
  Since $\Phi_{f,A}^k(D)=K^*[\Phi_{f,B}^k(I)\otimes I_\cD]K$,
  $\|\Phi_{f,B}^k(I)\|\leq 1$,
and $\Phi_{f,B}^k(I)\to 0$ strongly, as $k\to 0$, we deduce that
$D\in  C_{pure}(f,A)^+$. The proof is complete.
\end{proof}

We remark that, in Theorem \ref{poisson2}, the set $\cP$  is of
arbitrary noncommutative polynomials with $\cN_\cP\neq \{0\}$,
while, in Theorem \ref{poisson}, $\cP$ consists of homogeneous
polynomials.

The proof of the next result is similar to that of Corollary
\ref{VN}, so we shall omit it.
\begin{corollary}\label{VN2} Under  the hypotheses of Theorem \ref{poisson2}, if
$D\in C_{pure}(f,A)^+$, then
  we have the following von
Neumann type inequality:
\begin{equation*}
  \left\|\sum_{\alpha,\beta\in \Lambda} A_\alpha D
A_\beta^*\otimes C_{\alpha,\beta}\right\| \leq \|D\|
\left\|\sum_{\alpha,\beta\in \Lambda}
 B_\alpha B_\beta^*\otimes C_{\alpha,\beta}\right\|
 \end{equation*}
 for any
  finite  set $\Lambda\subset \FF_n^+$ and $C_{\alpha,\beta}\in
  B(\cE)$, where $\cE$ is a Hilbert space.

   If, in addition, $D$ is
  an invertible operator, then  the polynomial calculus
  $
  p(B_1,\ldots, B_n)\mapsto p(A_1,\ldots, A_n)
  $
  extends to a completely bounded map
$u:
  F_n^\infty(\cV^m_{f,\cP})\to B(\cH)$ by setting
  $$
 u(\varphi) :=K_q^* [\varphi \otimes I_\cH]
  K_q D^{-1},\qquad  \varphi \in F_n^\infty(\cV^m_{f,\cP}),
$$
where  \ $K_{q}$  is the noncommutative Berezin kernel  associated
with the compatible tuple $q:=(f,m, A,R,\cP)$ and
$R:=(id-\Phi_{f,A})^m(D)$.  Moreover, $\|u\|_{cb}\leq
\|D^{-1/2}\|\|D^{1/2}\|$.
\end{corollary}

\begin{theorem}\label{pure}
Let  $f:=\sum_{ |\alpha|\geq 1} a_\alpha X_\alpha$
 be   a positive regular  free holomorphic function and $m\geq 1$.
Let $\cP$ be a family of noncommutative
 polynomials  and let  $A:=(A_1,\ldots,
A_n)\in B(\cH)^n$ be  such that $\sum_{|\alpha|\geq 1} a_\alpha
A_\alpha A_\alpha^* $ \ is SOT-convergent and $p(A_1,\ldots, A_n)=0$
for $p\in \cP.$
 Then   a positive operator  $G\in B(\cH)$  is in $C (f,A)^+$
if and only
  if there exists  an $n$-tuple $T:=(T_1,\ldots, T_n)\in \cV^m_{f,\cP}(\cH)$
  such that
 \begin{equation*}
 A_i G^{1/2}= G^{1/2} T_i, \qquad i=1,\ldots, n.
 \end{equation*}
In addition,  $G\in C_{pure}(f,A)^+$    if and only if  \ $I_\cH \in
C_{pure}(f,T)^+$.
\end{theorem}
\begin{proof}
First, assume that  $T:=(T_1,\ldots, T_n)\in \cV^m_{f,\cP}(\cH)$
satisfies
  $
 A_i G^{1/2}= G^{1/2} T_i$,  for any $i=1,\ldots, n$.
 Then  we have
 $$
  (id-\Phi_{f,A})^s(G)= G^{1/2}\left[(id-\Phi_{f,T})^s(I)\right] G^{1/2}
\geq 0,\qquad s=1,\ldots,m.
 $$
Taking into account that $ \Phi_{f,A}^k(G)=G^{1/2}  \Phi_{f,T}^k(I)
G^{1/2} $, $k\in \NN$, it is clear that if $ \Phi_{f,T}^k(I)\to 0$
strongly, as $k\to\infty$,  then $G\in C_{pure}(f,A)^+$ .

 To prove the converse, assume that $G\in B(\cH)$  is in $C (f,A)^+$.
   Since
\begin{equation*}
\begin{split}
\sum_{|\alpha|\geq 1} \|G^{1/2} \sqrt{a_\alpha} A_\alpha^* x\|^2
=\left< \Phi_{f, A}(G)x,x\right>\leq \|G^{1/2} x\|^2
\end{split}
\end{equation*}
for any $x\in   \cH$, we deduce that  $a_{g_i}\|G^{1/2}A_i^*
x\|^2\leq \|G^{1/2} x\|^2$, for any $x\in   \cH$. Recall that
$a_{g_i}\neq 0$,  so we can define the operator
 $\Lambda_i:  G^{1/2}(\cH)\to
   G^{1/2}(\cH) $  by setting
 \begin{equation}\label{GiC}
 \Lambda_i G^{1/2}x:= G^{1/2} A_i^*x, \qquad  x\in \cH,
 \end{equation}
 for $i=1,\ldots, n$.
It is obvious that $\Lambda_i$
   can    be extended to a bounded operator (also
denoted by $\Lambda_i$) on the subspace $\cM:=\overline{
G^{1/2}(\cH)}$. Set $M_i:=\Lambda_i^*$, $i=1,\ldots, n$, and note
that
$$
G^{1/2}\left[(id-\Phi_{f,M})^s(I_\cM)\right] G^{1/2}=
(id-\Phi_{f,A})^s(G)\geq 0, \qquad s=1,\ldots,m.
$$
An approximation argument shows that
$$
(id-\Phi_{f,M})^s(I_\cM)  \geq 0,\qquad s=1,\ldots,m.
$$
Define $T_i:= M_i\oplus 0$, $i=1,\ldots, n$,
  with respect to the decomposition
 $\cH= \cM\oplus \cM^\perp$,
   and note that
 $(id-\Phi_{f,T})^s(I)\geq 0$, $s=1,\ldots,m$.
Due to relation \eqref{GiC}, if $p\in \cP$, then we have
$$
p(M_1,\ldots,M_n)^*G^{1/2} =G^{1/2} p(A_1,\ldots, A_n)^*=0.
$$
Hence,  $p(M_1,\ldots,M_n)=0$ and, consequently,   $p(T_1,\ldots,
T_n)=0$ for all $p\in \cP$. Therefore,  $(T_1,\ldots, T_n)\in
\cV^m_{f,\cP}(\cH)$ and  $A_i G^{1/2}= G^{1/2} T_i$,  $i=1,\ldots,
n$.

 Assume  now that  $G\in C_{pure}(f,A)^+$, i.e., $ \Phi_{f,A}^k(G)\to 0$ strongly,
   as $k\to\infty$.
Since
\begin{equation*}
\begin{split}
\left< \Phi_{f, T}^k(I) G^{1/2}x, G^{1/2} x\right> &= \left<
\Phi_{f,
 A}^k(G)x,  x\right>,\qquad x\in   \cH,
\end{split}
\end{equation*}
  we have \ SOT-$\lim_{k\to\infty}\Phi_{f,T}^k(I)y=0$ for
any $y\in \text{\rm range}~G^{1/2}$. Taking into account that
$\|\Phi_{f,T}^k(I)\|\leq 1$, $k\in \NN$,  an approximation argument
shows that SOT-$\lim_{k\to\infty}\Phi_{f,T}^k(I)y=0$ for any $y\in
\overline {G^{1/2}(\cH)}$. On the other hand, we have
$\Phi_{f,T}^k(I)z=0$ for any $z\in \cM^\perp$. Consequently,
 $I_\cH \in
C_{pure}(f,T)^+$.
 This completes the proof.
\end{proof}

In what follows  we consider  the case when  $m=1$.  Let
  $f :=\sum_{|\alpha|\geq 1} a_\alpha X_\alpha$
   be
a positive regular free holomorphic function and let $\cP$ be a
family of noncommutative
 polynomials such that $\cN_\cP\neq \{0\}$.
We  have
$$
 {\bf D}^1_f(\cH):=\{(X_1,\ldots, X_n)\in B(\cH)^n: \
\sum_{|\alpha|\geq 1} a_\alpha X_\alpha X_\alpha^* \leq I \}.
$$
  Let  $B:=(B_1,\ldots,
B_n)$ be the universal
  model associated with the noncommutative variety
 $\cV^1_{f,\cP}$.
We introduced in \cite{Po-domains} the
 noncommutative Hardy algebra
 $F_n^\infty(\cV^1_{f,\cP})$  to be the $w^*$-closed algebra generated by
$B_1,\ldots, B_n$ and the identity. We also  showed that $
F_n^\infty(\cV^1_{f,\cP})=P_{\cN_\cP}F_n^\infty({\bf D}^1_f)
|_{\cN_\cP}$. Similar results hold for $R_n^\infty(\cV^1_{f,\cP})$,
the $w^*$-closed algebra generated by $C_1,\ldots, C_n$ and the
identity, where $C_i:=P_{\cN_\cP} \Lambda_i |_{\cN_\cP}$, and
$\Lambda_1,\ldots, \Lambda_n$ are the weighted right creation
operators associated with ${\bf D}^1_f$ (see Section 1).
 Moreover,  we proved  that
\begin{equation*}
F_n^\infty(\cV^1_{f,\cP})^\prime=R_n^\infty(\cV^1_{f,\cP})\ \text{
and } \ R_n^\infty(\cV^1_{f,\cP})^\prime=F_n^\infty(\cV^1_{f,\cP}),
\end{equation*}
where $^\prime$ stands for the commutant. An operator $M\in
B(\cN_{\cP}\otimes \cK,\cN_{\cP}\otimes \cK')$ is called
multi-analytic with respect to the constrained  weighted shifts $B_1,\ldots,
B_n$ if
$$
M(B_i\otimes I_{\cK})=(B_i\otimes I_{\cK'})M,\qquad i=1,\ldots, n.
$$
   According to \cite{Po-domains}, the set of all multi-analytic operators with respect to
 $B_1,\ldots, B_n$ coincides  with
$$
R_n^\infty(\cV^1_{f,\cP})\bar\otimes B(\cK,\cK')=P_{\cN_\cP\otimes
\cK'}[R_n^\infty({\bf D}^1_f)\bar\otimes
B(\cK,\cK')]|_{\cN_\cP\otimes \cK},
$$
and a similar result holds for  the  Hardy algebra
$F^\infty_n(\cV^1_{f,\cP})$. For more information on multi-analytic
operators, we refer the reader to \cite{Po-analytic} and
\cite{Po-domains}.

 \begin{theorem}\label{Adv}
Let $\cP$ be a family of noncommutative
 polynomials  with $\cN_\cP\neq \{0\}$ and let  $B:=(B_1,\ldots, B_n)$ be
the universal
  model associated with the noncommutative variety
 $\cV^1_{f,\cP}$, where  $f:=\sum_{ |\alpha|\geq 1} a_\alpha X_\alpha$
 is  a positive regular  free holomorphic function.
If $T:=(T_1,\ldots, T_n)$ is a pure
 $n$-tuple of operators
  in the noncommutative variety  $ \cV^1_{f,\cP}(\cH)$,
then $$C (f,T)^+ = C_{pure} (f,T)^+$$
  and any operator in $C (f,T)^+$ has the  form
    $G=P_\cH \Psi \Psi^*|_\cH $, where $\Psi$ is a multi-analytic operator
     with respect to $B_1,\ldots, B_n$.
 \end{theorem}
 \begin{proof} Assume that $T:=(T_1,\ldots, T_n)$ is a pure
 $n$-tuple of operators
  in the noncommutative variety  $ \cV^1_{f,\cP}(\cH)$, i.e.,
   $\Phi_{f,T}^k(I)\to 0$ strongly, as $k\to \infty$. If $G\in C (f,T)^+$, then $G\geq 0$ and
 $ \Phi_{f,T}(G)\leq  G$.   Since
  $$0\leq \Phi_{f,T}^k(G)\leq \|G\|\, \Phi_{f,T}^k(I),\qquad k=1,2,
  \ldots,
  $$
    we infer that
   $ G\in C_{pure} (f,T)^+$. Consequently, we have  $C (f,T)^+ = C_{pure}
   (f,T)^+$.
Now, fix an operator  $G\in C_{pure}
   (f,T)^+$.
  Due to  Theorem \ref{pure}, we
find $D_i\in B(\cH)$ satisfying
$$
T_i G^{1/2}= G^{1/2} D_i, \qquad i=1,\ldots, n,
$$
  where $(D_1,\ldots, D_n)\in \cV^1_{f,\cP}(\cH)$  and  $\Phi_{f,D}^k(I)\to 0$ strongly,  as $ k\to
  0$.
According  to Theorem 3.20 from \cite{Po-domains}, there is a
Hilbert space $\cM_1$ so that $(B_1\otimes I_{\cM_1},\ldots,
B_n\otimes I_{\cM_1})$ is a dilation of $(T_1,\ldots, T_n)$ on the
Hilbert space $\cK_1:= \cN_\cP\otimes \cM_1\supseteq \cH$, i.e.,
$T_i=P_\cH(B_i\otimes I_{\cM_1})|_{\cH}$, $i=1,\ldots, n$, and $\cH$
is invariant under each operator $B_i^*\otimes I_{\cM_1}$.
 Similarly,
   let $(B_1\otimes I_{\cM_2},\ldots, B_n\otimes I_{\cM_2})$  be a dilation
   of $(D_1,\ldots, D_n)$  on a Hilbert space $\cK_2:=\cN_\cP\otimes \cM_2 \supseteq \cH$.
 According to the noncommutative commutant lifting theorem from
 \cite{Po-domains} (see Theorem 4.2), there exists an operator
  $\widehat{G}:\cK_2\to \cK_1$
 such that $ \widehat{G}^*(\cH)\subset \cH$,  $
 \widehat{G}^*|_\cH=G^{1/2}$,
  \ $\|\widehat{G}\|=\|G^{1/2}\|$, and
 $$
 \widehat{G}^*(B_i^*\otimes I_{\cM_1}) =(B_i^*\otimes I_{\cM_2})\widehat{G}^*, \qquad  i=1,\ldots, n.
 $$
 It is easy to see that
 $$
 \Phi_{f, B\otimes I_{\cM_1}}( \widehat{G}  \widehat{G}^*)=
  \widehat{G}   \Phi_{f, B\otimes I_{\cM_2}}(I)\widehat{G}^*\leq \widehat{G}
  \widehat{G}^*.
 $$
 Setting $Q:= \widehat{G}  \widehat{G}^*$, we have   \ $\|Q\|=\|G\|$,
   and  $$G = P_\cH \widehat G|_\cH G^{1/2}=P_\cH \widehat G \widehat G^*|_\cH=P_\cH
   Q|_\cH.
   $$
  Note also that
  $$
  \Phi_{f, B\otimes I_{\cM_1}}^k( \widehat{G} \widehat{G}^*)=
  \widehat{G}   \Phi_{f, B\otimes I_{\cM_2}}^k(I)\widehat{G}^*,\qquad k\in \NN.
  $$
Since $\Phi_{f, B\otimes I_{\cM_2}}^k(I)\to 0$ strongly, as $k\to
 \infty$, we deduce that $ \Phi_{f, B\otimes I_{\cM_1}}^k( \widehat{G}
 \widehat{G}^*)\to 0$ strongly. Therefore, $Q\in C_{pure}(f, B\otimes
 I_{\cM_1})^+$
 and $G=P_\cH
   Q|_\cH.$

Conversely, if $Q\in C_{pure}(f, B\otimes
 I_{\cM_1})^+$,
 then
 \begin{equation*}\begin{split}
 \Phi_{f,T}(P_\cH Q|_\cH)&= \sum_{|\alpha|\geq 1}   a_\alpha T_\alpha (P_\cH Q|_\cH)T_\alpha^*\\
 &=
 P_\cH [ \Phi_{f, W\otimes I_{\cM_1}}(P_\cH Q|_\cH)] |_\cH \\
 &\leq P_\cH [ \Phi_{f, W\otimes I_{\cM_1}}( Q)] |_\cH\\
 &\leq P_\cH Q|_\cH.
 \end{split}
 \end{equation*}
On the other hand,  since
 $$
0\leq  \Phi_{f,T}^k(P_\cH Q|_\cH)\leq  P_\cH \Phi_{f,B\otimes
 I_{\cM_1}}^k(Q)|_\cH\to 0,\quad \text{ as } k\to\infty,
$$
  it is clear that $G:=P_\cH Q|_\cH$ is  in $  C_{pure}
   (f,T)^+$. We have proved that
   $$
   C_{pure}
   (f,T)^+= P_\cH\left[C_{pure}(f, B\otimes
 I_{\cM_1})^+\right]|_\cH.
 $$

Now, we determine the set $C_{pure}(f, B\otimes
 I_{\cM_1})^+$. To this end, let  $Q\in C_{pure}(f, B\otimes
 I_{\cM_1})^+$.
    According to  Theorem \ref{poisson2},  $Q\in C_{pure}(f, B\otimes
 I_{\cM_1})^+$ if and only if there is a Hilbert space $\cD$ and an
 operator
 $K:\cN_\cP\otimes \cM_1\to \cN_\cP\otimes \cD$ such that  $Q=K^*K$
 and
$$
    (B_i\otimes I_{\cM_1}) K^*= K^* (B_i\otimes I_\cD), \qquad i=1,\ldots, n,
    $$
    i.e., $K^*$ is a multi-analytic operator  with respect to $B_1,\ldots, B_n$.
 The proof is complete.
 \end{proof}

\bigskip

\section{ Joint similarity   to operators in noncommutative varieties }\label{Similarity}

In  this section  we
 provide necessary and sufficient conditions
for  an $n$-tuple  $A:=(A_1,\ldots, A_n)\in B(\cH)^n$ to
be jointly similar  to an $n$-tuple $T:=(T_1,\ldots, T_n)$
satisfying one of the following properties:
\begin{enumerate}
\item[(i)]
$T\in {\cV}_{f,\cP}^m(\cH)$;
\item[(ii)]
$T\in  \left\{ X \in {\cV}_{f,\cP}^m(\cH)): \
 (id-\Phi_{f,X})^m(I)= 0\right\}$;
\item[(iii)] $T\in \left\{ X \in {\cV}_{f,\cP}^m(\cH)): \
 (id-\Phi_{f,X})^m(I)>0\right\}$;
\item[(iv)]  $T$ is a pure  $n$-tuple in ${\cV}_{f,\cP}^m(\cH)$, i.e., $\Phi_{f,T}^k(I)\to 0$ strongly, as $k\to \infty$,
\end{enumerate}
where $\cP$ is a set of noncommutative polynomials.
We show that these similarities are strongly related to the
existence of invertible positive solutions of the operator
inequality $(id-\Phi_{f,A})^m(Y)\geq 0$ or equation  $
(id-\Phi_{f,A})^m(Y)= 0$. Several classical results concerning  the similarity
 to contractions have  analogues  in our  multivariable setting.

Let $f=\sum_{|\alpha|\geq 1 } a_\alpha X_\alpha$ be a positive
regular free holomorphic. For any $n$-tuple of operators
$A:=(A_1,\ldots, A_n)\in B(\cH)^n$   such that $\sum_{|\alpha|\geq
1} a_\alpha A_\alpha A_\alpha^* $ is convergent in the weak operator
topology,  define the joint spectral radius with respect to the
noncommutative domain ${\bf D}^m_f$ by setting
$$
r_f(A_1,\ldots, A_n):=\lim_{k\to\infty}\|\Phi_{f,A}^k(I)\|^{1/2k}.
$$
In the particular case when $f:=X_1+\cdots +X_n$, we obtain the
usual definition of the joint operator radius for $n$-tuples of
operators.

Our first result provides  necessary conditions for joint similarity
to $n$-tuples of operators in noncommutative varieties
$\cV_{f,\cP}^m(\cH)$.

\begin{proposition} \label{properties}  Let $f:=\sum_{|\alpha|\geq 1} a_\alpha X_\alpha$ be
a positive regular free holomorphic function and let $T:=(T_1,\ldots,
T_n)\in B(\cH)^n$ and
 $A:=(A_1,\ldots,
A_n)\in B(\cK)^n$ be  two $n$-tuples of operators which are jointly
similar, i.e., there exists an invertible operator $Y:\cH\to \cK$
such that
$$
A_i=Y T_i Y^{-1},\qquad i=1,\ldots,n.
$$
If  $\cP$ is a family of noncommutative
 polynomials  and
 $T\in \cV_{f,\cP}^m(\cH)$, then the following statements hold:
  \begin{enumerate}
  \item[(i)]$\sum_{|\alpha|\geq 1 } a_\alpha A_\alpha A_\alpha^*$ is convergent in the weak operator topology;
       \item[(ii)] $\Phi_{f,A}$ is a power bounded completely positive linear map;
      \item[(iii)] $r_f(A_1,\ldots, A_n)\leq 1$;
     \item[(iv)] $p(A_1,\ldots, A_n)=0$ for all $p\in \cP$;
     \item[(v)] if $\Phi_{f,T}^k(I)\to 0$
 strongly, as $k\to \infty$, then $\Phi_{f,A}^k(I)\to 0$
 strongly.
      \end{enumerate}
\end{proposition}
\begin{proof} Note that
\begin{equation*}
\begin{split}
\sum_{1\leq|\alpha|\leq k} a_\alpha A_\alpha A_\alpha^* &=
\sum_{1\leq|\alpha|\leq k} a_\alpha Y T_\alpha Y^{-1}
(Y^{-1})^*T_\alpha^* Y^*\\
&\leq \|Y^{-1}\|^2\, Y\left(\sum_{1\leq|\alpha|\leq k} a_\alpha
T_\alpha  T_\alpha^*\right) Y^*
\end{split}
\end{equation*}
for any $k\in \NN$. Since  $T:=(T_1,\ldots, T_n)\in
\cV_{f,\cP}^m(\cH)$, the series $\sum_{|\alpha|\geq 1} a_\alpha
T_\alpha T_\alpha^*$ is convergent in the weak operator topology and
$p(T_1,\ldots, T_n)=0$ for all $p\in \cP$. Now, due to inequality
above, it is easy to see that item (i) holds and
$$
\Phi_{f,A}(I)\leq \|Y^{-1}\|^2 \,Y\Phi_{f,T}(I) Y^*\leq
\|Y^{-1}\|^2\| Y\|^2 I.
$$
According to Lemma \ref{wot-cp}, $\Phi_{f,A}$ is a completely
positive map. As  above, one can also show that
$$
\Phi^k_{f,A}(I)\leq \|Y^{-1}\|^2 \,Y\Phi^k_{f,T}(I) Y^*\leq
\|Y^{-1}\|^2\| Y\|^2 I,\qquad k\in \NN,
$$
which proves item (ii) and implies items (iii) and (v). Since item
(iv) is obvious, the proof is complete.
\end{proof}

 We recall that $C (f,A)^+$  is the cone of all positive operators $D\in B(\cH)$ such that  $(id-\Phi_{f, A})^s(D)\geq 0$ for  $s=1,\ldots, m$.
  Now, we are ready to  provide necessary and sufficient conditions for the joint similarity to
  parts of the adjoints
  of  the universal model $(B_1,\ldots, B_n)$  associated with the
noncommutative variety $\cV^m_{f,\cP}$.

\begin{theorem}\label{pure-contr} Let $m\geq 1$,  $f:=\sum_{|\alpha|\geq 1} a_\alpha X_\alpha$ be
a positive regular free holomorphic function and let $\cP$ be a family of
noncommutative
 polynomials   with $\cN_\cP\neq 0$. If  $A:=(A_1,\ldots,
A_n)\in B(\cH)^n$  is such that $\sum_{|\alpha|\geq 1} a_\alpha
A_\alpha A_\alpha^* $ is convergent in the weak operator topology
and $p(A_1,\ldots, A_n)=0$,\quad $p\in \cP$, then the following
statements are equivalent.
\begin{enumerate}
\item[(i)] There exists an invertible operator $Y:
\cH\to \cG$ such that
$$
A_i^*=Y^{-1}[(B_i^*\otimes I_\cH)|_\cG]Y,\qquad i=1,\ldots, n,
$$
where
 $\cG\subseteq \cN_\cP\otimes \cH$ is an invariant  subspace under  each operator $B_i^*\otimes
 I_\cH$  and $(B_1,\ldots, B_n)$ is
the universal
  model associated with the noncommutative variety
 $\cV^m_{f,\cP}$.
 \item[(ii)] There is   an
 invertible operator $Q\in C(f,A)^+$ such that  $ \Phi_{f,A}^k(Q)\to 0$ strongly, as $k\to \infty$.
\item[(iii)]  There exist
constants $0<a\leq b$ and a positive operator $R\in B(\cH)$ such
that
$$
aI\leq \sum_{k=0}^\infty \left(\begin{matrix} k+m-1\\ m-1
\end{matrix}\right)\Phi_{f,A}^k(R)\leq bI.
$$
\end{enumerate}
\end{theorem}

\begin{proof} We prove that (i) $\Rightarrow$ (ii).  Assume
that (i) holds and let $a,b>0$ be such that $aI\leq Y^*Y\leq bI$.
Setting $Q:=Y^*Y$ and using the fact that $\Phi_{f,B}(I)\leq I$ and
$a_\alpha \geq 0$, we have
\begin{equation*}
\begin{split}
\Phi_{f,A}(Q)&= \sum_{|\alpha|\geq 1} a_\alpha Y^*[P_\cG(B_\alpha
B_\alpha^*\otimes I_\cH)|_\cG]Y\\
&=Y^*\left\{P_\cG\left[\sum_{|\alpha|\geq 1} a_\alpha  B_\alpha
B_\alpha^*\otimes I\right]|_\cG \right\}Y\\
&\leq Y^*Y=Q.
\end{split}
\end{equation*}
Similar calculations reveal that
$$
(id-\Phi_{f,A})^s(Q)=Y^*\left\{P_\cG\left[ (id-\Phi_{f,A})^s(I)\otimes I\right]|_\cG \right\}Y\geq 0,\qquad  s=1,\ldots,m.
$$
Therefore, $Q\in C(f,A)^+$. Since $(B_1,\ldots, B_n)$ is a pure
$n$-tuple in the noncommutative variety $\cV^m_{f,\cP}$, we have
$\Phi_{f,B}^k(I)\to 0$ strongly, as $k\to\infty$. Taking into
account that $\Phi_{f,A}^k(Q)= Y^*\left[P_\cG(\Phi_{f,B}^k(I)\otimes
I)|_\cG\right]Y$ for $k\in \NN$, we deduce that
 $\Phi_{f,A}^k(Q)\to 0$
strongly, as $k\to\infty$. Therefore   item  (ii) holds.

Now, we prove the implication  (ii) $\Rightarrow$ (iii).   Let $Q\in C(f,A)^+$
be an  invertible operator such that
  $\Phi_{f,A}^k(Q)\to 0$ strongly, as $k\to\infty$. Set
$R:=(id-\Phi_{f,A})^m(Q)$ and note  that, using  Lemma
\ref{ineq-lim}, we obtain
\begin{equation*}\begin{split}
\sum_{k=0}^\infty \left(\begin{matrix} k+m-1\\ m-1
\end{matrix}\right) \Phi_{f,A}^k(R )&=
Q- \text{\rm SOT-}\lim_{k\to \infty}\sum_{j=0}^{m-1}\left(\begin{matrix} k+j\\ j
\end{matrix}\right)\Phi_{f,A}^{k+1}(id-\Phi_{f,A})^j(Q)\\
&= Q-\text{\rm SOT-}\lim_{k\to \infty}  \Phi_{f,A}^k(Q)=Q.
\end{split}
\end{equation*}
 Hence, we deduce item (iii).
It remains to show that    (iii) $\Rightarrow$ (i). Assume that item
(iii) holds. Consider the noncommutative  Berezin  kernel
$K^{(m)}_{f,A,R}:\cH\to F^2(H_n)\otimes  \cH$ associated with with
the quadruple $(f,m,A,R)$ and
   defined by
\begin{equation*}
 K^{(m)}_{f,A,R}h=\sum_{\alpha\in \FF_n^+} \sqrt{b^{(m)}_\alpha}
e_\alpha\otimes   R^{1/2} A_\alpha^* h,\qquad h\in \cH.
\end{equation*}
According to Lemma \ref{lemma1} and using item (iii), we have
\begin{equation}
\label{K-norm} a\|h\|^2\leq \|K^{(m)}_{f,A,R}h\|^2=\sum_{k=0}^\infty
\left(\begin{matrix} k+m-1\\ m-1
\end{matrix}\right)
\langle\Phi_{f,A}^k(R)h,h\rangle\leq b\|h\|^2,\qquad h\in \cH.
\end{equation}
 Consequently,  the range of $ K^{(m)}_{f,A,R}$
is a closed subspace of $F^2(H_n)\otimes \cH$. On the other hand, we
showed in Lemma \ref{lemma2} that
\begin{equation*}
\text{\rm range}\,K^{(m)}_{f,A,R}\subseteq \cN_\cP\otimes \overline{
R^{1/2}(\cH)}
\end{equation*}
 and    the noncommutative
Berezin kernel  $K_{q}:\cH\to \cN_\cP\otimes \overline{
R^{1/2}(\cH)}$ associated with the  compatible tuple
$q:=(f,m,A,R,\cP)$  and
 defined by
$$K_{q}:=\left(P_{\cN_\cP}\otimes I_{\overline{ R^{1/2}(\cH)}}\right)K^{(m)}_{f,A,R},
$$
 has the property that
\begin{equation*}
K_{q}A_i^*=(B_i^*\otimes I_{\overline{ R^{1/2}(\cH)}})K_{q},\qquad
i=1,\ldots,n.
\end{equation*}
Consequently, the range of $K_{q}$ is a closed subspace  of
$\cN_\cP\otimes \cH$ and it is an invariant  subspace under each
operator $B_i^*\otimes I_\cH$, $i=1,\ldots,n$.
  Since  the operator $Y:\cH\to \text{\rm range}\,
K_{q}$ defined by $Yh:=K_{q}h$, $h\in \cH$,  is invertible, we have
\begin{equation}\label{WTW}
A_i^*= Y^{-1}[(B_i^*\otimes I_{\overline{
R^{1/2}(\cH)}})|_\cG]Y,\qquad i=1,\ldots, n,
\end{equation}
where $\cG:=\text{\rm range}\, K_{q}$.
   This proves (i).
   The proof is complete.
\end{proof}

We remark that under the conditions of Theorem \ref{pure-contr},
part (iii), we can use relations \eqref{K-norm} and \eqref{WTW} to
show that the map $\Psi:\cA_n(\cV^m_{f,\cP})\to B(\cH)$ defined by
$$\Psi(p(B_1,\ldots, B_n)):=
 p(A_1,\ldots, A_n)$$  is completely bounded  with ~$\|\Psi\|_{cb}\leq
  \sqrt{ \frac {b} {a}}$,  and $\Phi_{f,A}^k(I)\to 0$
 strongly, as $k\to \infty$.
In the particular case when $m=1$
we have a converse of the latter result.
Indeed, using  Paulsen's  similarity result \cite{Pa}
   and the fact (which can be extracted from \cite{Po-domains}) that
   any completely contractive representation of the noncommutative variety algebra
   $\cA_n(\cV^1_{f,\cP})$ is generated  by an $n$-tuple
   $(T_1,\ldots, T_n)\in \cV^1_{f,\cP}(\cH)$, we infer that
   $(A_1,\ldots, A_n)$ is simultaneously similar to an $n$-tuple
   $(T_1,\ldots, T_n)\in \cV^1_{f,\cP}(\cH)$ and $\Phi_{f,T}^k(I)\to \infty$ strongly, as $k\to\infty$.
   This proves our assertion.

Taking $R=I$ in Theorem \ref{pure-contr}, we can obtain the
following
  analogue of Rota's  model  theorem,  for  similarity to $n$-tuples of operators  in the  noncommutative
variety $\cV^m_{f,\cP}(\cH)$.

\begin{corollary}
Let $\cP$ be a set of noncommutative
 polynomials   with $\cN_\cP\neq 0$ and let $A:=(A_1,\ldots, A_n)\in B(\cH)^n$ be such that
$p(A_1,\ldots, A_n)=0,\quad p\in \cP, $ and
$$
 \sum_{k=0}^\infty \left(\begin{matrix} k+m-1\\ m-1
\end{matrix}\right)\Phi_{f,A}^k(I)\leq bI
$$
for some constant  $ b>0$.  Then, there exists an invertible
operator $Y:  \cH\to \cG$ such that
$$
A_i^*=Y^{-1}[(B_i^*\otimes I_\cH)|_\cG]Y,\qquad i=1,\ldots, n,
$$
where
 $\cG\subseteq \cN_\cP\otimes \cH$ is an invariant under  each operator $B_i^*\otimes
 I_\cH$  and $(B_1,\ldots, B_n)$ is the universal
  model associated with the noncommutative variety
 $\cV^m_{f,\cP}$.
\end{corollary}

Another consequence of Theorem \ref{pure-contr} is
 the following analogue of Foia\c s \cite{Fo} and de Branges--Rovnyak \cite{BR}
  model theorem, for pure $n$-tuples of operators in $\cV^m_{f,\cP}(\cH)$.
\begin{corollary} \label{Fo-BR}
An $n$-tuple of operators $T:=(T_1,\ldots, T_n)\in B(\cH)^n$  is in the noncommutative
 variety $ \cV_{f,\cP}^m$ and it is pure, i.e., $\Phi_{f,T}^k(I)\to 0$ strongly, as $k\to \infty$, if and only if
there exists a unitary  operator $U:
\cH\to \cG$ such that
$$
T_i^*=U^*[(B_i^*\otimes I_\cD)|_\cG]U,\qquad i=1,\ldots, n,
$$
where
 $\cD:=\overline{\left[(id-\Phi_{f,T})^m(I)\right]^{1/2}(\cH)}$, the subspace
$\cG\subseteq \cN_\cP\otimes \cH$ is  invariant   under
each operator  $B_i^*\otimes
 I_\cD$  and $(B_1,\ldots, B_n)$ is
the universal
  model associated with the noncommutative variety
 $\cV^m_{f,\cP}$.
\end{corollary}
\begin{proof} Let $T:=(T_1,\ldots, T_n)\in  \cV_{f,\cP}^m(\cH)$ be
such that $\Phi_{f,T}^k(I)\to 0$ strongly, as $k\to \infty$. A
closer look at the proof of Theorem \ref{pure-contr}, when $A=T$ and
$Q=I_\cH$, reveals that
$$
K_q^* K_q=\sum_{k=0}^\infty \left(\begin{matrix} k+m-1\\ m-1
\end{matrix}\right) \Phi_{f,A}^k(R )=I,
$$
where $R:=(id-\Phi_{f,A})^m(I)$. Consequently, $K_q$ is an isometry
and   the operator   $U:\cH\to   {K_q(\cH)}$, defined by $Uh:=K_qh$,
$h\in \cH$, is unitary. Now, one can use  relation \eqref{WTW}  to
complete the proof.
\end{proof}

Next we obtain an analogue of  Sz.-Nagy's similarity result
\cite{SzN}.

\begin{theorem}\label{simi} Let $m\geq 1$,  $f:=\sum_{|\alpha|\geq 1} a_\alpha X_\alpha$ be
a positive regular free holomorphic function and let $\cP$ be a family of
noncommutative
 polynomials. If  $A:=(A_1,\ldots,
A_n)\in B(\cH)^n$  is such that $\sum_{|\alpha|\geq 1} a_\alpha
A_\alpha A_\alpha^* $ is convergent in the weak operator topology
and $p(A_1,\ldots, A_n)=0$,\quad $p\in \cP$, then the following
statements are equivalent.
\begin{enumerate}
\item[(i)] There exist $(T_1,\ldots, T_n)\in {\cV}_{f,\cP}^m(\cH)$ such that
$ (id-\Phi_{f,T})^m(I)= 0
$
and an invertible operator $Y\in B(\cH)$ such that
$$
A_i=Y^{-1} T_i Y,\qquad i=1,\ldots,n.
$$
\item[(ii)]
There exist
 positive constants
 $0<c\leq d$ such that
 $$
 cI\leq  \Phi_{f,A}^k(I)\leq dI, \qquad k\in \NN.
 $$
 \item[(iii)] $\Phi_{f,A}$ is power bounded and
 there exists  an invertible positive operator $Q\in B(\cH)$
 such that  $(id-\Phi_{f,A})^m(Q)= 0$.
\end{enumerate}
\end{theorem}

\begin{proof}
First we prove that (i) $\Leftrightarrow$ (ii). Assume item (i) holds.
Then we have
\begin{equation*}
\begin{split}
\Phi_{f,A}^k(I)&=Y^{-1} \Phi_{f,T}^k(YY^*) {Y^*}^{-1}\\
&\leq \|YY^*\| Y^{-1} \Phi_{f,T}^k(I) {Y^*}^{-1}\\
&\leq \|Y\|^2\|Y^{-1}\|^2 I
\end{split}
\end{equation*}
for any $k\in \NN$. Now, we show that $\Phi_{f,T}(I)=I$. As in the
proof of Theorem \ref{pure-contr}, we have
\begin{equation*}
\sum_{p=0}^q\left(\begin{matrix} p+m-1\\m-1
\end{matrix}\right) \Phi_{f,T}^p(id-\Phi_{f,T})^m(I)=I-\sum_{j=0}^{m-1}
 \left(\begin{matrix} q+j\\ j
\end{matrix}\right) \Phi_{f,T}^{q+1} (id-\Phi_{f,T})^j(I)
\end{equation*}
for any $q\in \NN$. Consequently, if $(id-\Phi_{f,T})^m(I)=0$,  then
$$I=\lim_{q\to \infty} \sum_{j=0}^{m-1}
 \left(\begin{matrix} q+j\\ j
\end{matrix}\right) \Phi_{f,T}^{q+1} (id-\Phi_{f,T})^j(I).
$$
Using Lemma \ref{ineq-lim}, we deduce that $I=\lim_{q\to \infty} \Phi_{f,T}^q(I)$.
Since $\Phi_{f,T}$ is a positive linear map and $\Phi_{f,T}(I)\leq I$, we have
$$
I=\lim_{q\to \infty} \Phi_{f,T}^q(I)\leq \ldots \leq \Phi_{f,T}^2(I)\leq \Phi_{f,T}(I)\leq I.
$$
Hence, we deduce that $\Phi_{f,T}(I)=I$.
Consequently, we have
\begin{equation*}
\begin{split}
I=\Phi_{f,T}^k(I)&=Y \Phi_{f,A}^k(Y^{-1}{Y^*}^{-1}) {Y^*} \\
&\leq \|Y^{-1}{Y^*}^{-1}\| \,Y  \Phi_{f,A}^k(I) Y^*
\end{split}
\end{equation*}
for any $k\in\NN$.
Hence, we deduce that
$$
Y^{-1}{Y^*}^{-1}\leq \|Y^{-1}\|^2\, \Phi_{f,A}^k(I),\qquad k\in\NN,
$$
 which implies
$$
\Phi_{f,A}^k(I)\geq \frac{1}{\|Y^{-1}\|^2} Y^{-1} {Y^*}^{-1}\geq
\frac{1}{\|Y\|^2 \|Y^{-1}\|^2}I.
$$
Therefore, we have proved that
$$
\frac{1}{\|Y\|^2 \|Y^{-1}\|^2}I\leq \Phi_{f,A}^k(I) \leq
\|Y\|^2\|Y^{-1}\|^2 I,\qquad k\in\NN.
$$
Therefore, item (ii) holds.
We prove now  the implication  (ii) $\Rightarrow$ (iii).
 Assume  that item (ii) holds.
For each $k\geq 1$, we define the operator
$$
Q_k:=  \frac {1} {k} \sum_{j=0}^{k-1} \Phi_{f,A}^j(I)
$$
and note that  $cI\leq Q_k\leq dI$.
Since the closed unit ball of $B(\cH)$ is weakly compact, there is a subsequence
  $\{Q_{k_j}\}_{j=1}^\infty$
 weakly
 convergent   to an operator $Q\in B(\cH)$.
 It is clear that $Q$ is an invertible positive
 operator and
 $aI\leq Q\leq bI$.
 Since
 $$Q_{k_j}-\Phi_{f,A}(Q_{k_j})= \frac {1} {k_j} I- \frac {1} {k_j} \Phi_{f,A}^{k_j}(I)
 $$
 and taking into account  that $\frac {1} {k_j} \Phi_{f,A}^{k_j}(I)\to 0$ in norm as $j\to\infty$ , we get $\|Q_{k_j}-\Phi_{f,A}(Q_{k_j})\|\to 0$,
  as $j\to\infty$.
 On the other hand, according to Lemma \ref{wot-cp}, $\Phi_{f,A}$ is WOT-continuous on bounded sets.
  Now, using the fact that $Q_{k_j}$ converges weakly to $Q$, we deduce that $\Phi_{f,A}(Q)= Q$, which implies  $(id-\Phi_{f,A})^m(Q)= 0$ and shows that
    item  (iii) holds.

  It remains to show that (iii) $\Rightarrow $ (i). Assume that $\Phi_{f,A}$ is power bounded and
 there exists  an invertible positive operator $Q\in B(\cH)$
 such that such that  $(id-\Phi_{f,A})^m(Q)= 0$.
 Since
\begin{equation*}
\sum_{p=0}^q\left(\begin{matrix} p+m-1\\m-1
\end{matrix}\right) \Phi_{f,A}^p(id-\Phi_{f,A})^m(Q)=Q-\sum_{j=0}^{m-1}
 \left(\begin{matrix} q+j\\ j
\end{matrix}\right) \Phi_{f,A}^{q+1} (id-\Phi_{f,A})^j(Q)
\end{equation*}
for any $q\in \NN$, we deduce that
$$Q=\lim_{q\to \infty} \sum_{j=0}^{m-1}
 \left(\begin{matrix} q+j\\ j
\end{matrix}\right) \Phi_{f,A}^{q+1} (id-\Phi_{f,A})^j(Q).
$$
Using Lemma \ref{ineq-lim}, we deduce that $Q=\lim_{q\to \infty} \Phi_{f,A}^q(Q)$.

On the other hand, since
 $\Phi_{f,A}$ is a power bounded positive linear map with $(id-\Phi_{f,A})^m(Q)\geq 0$,
 we can use Lemma  \ref{wot-cp} to deduce that
 $(id-\Phi_{f,A})^s(Q)\geq 0$ for  any $s=1,\ldots, m$.
 In paticular, we have
  $\Phi_{f,A}(Q)\leq Q$. Using the results above, we have
$$
Q=\lim_{q\to \infty} \Phi_{f,A}^q(Q)\leq \ldots \leq
\Phi_{f,A}^2(Q)\leq \Phi_{f,A}(Q)\leq Q.
$$
Hence, we deduce that $\Phi_{f,A}(Q)=Q$.
 Set $T_i:=Q^{-1/2} A_i Q^{1/2}$ for  $i=1,\ldots,n$ and note
  that
  \begin{equation*}
  \begin{split}
\sum_{|\alpha|\geq 1} a_\alpha T_\alpha T_\alpha^*
&=Q^{-1/2}\left(\sum_{|\alpha|\geq 1} a_\alpha A_\alpha
QA_\alpha^*\right)Q^{-1/2}\\
&=Q^{-1/2} Q Q^{-1/2}=I,
  \end{split}
  \end{equation*}
which implies item (i). The proof is complete.
\end{proof}

Now, we can obtain  a noncommutative multivariable analogue  of
Douglas' similarity result \cite{Do}.
\begin{corollary}\label{inve} If  $A:=(A_1,\ldots, A_n)\in B(\cH)^n$  satisfies the conditions of Theorem \ref{simi} and
\begin{equation*}
\Phi_{f,A}^\infty(I):= \text{\rm SOT-}\lim_{k\to \infty}
\Phi_{f,A}^k(I)
\end{equation*}
exists, then
 the following statements
are equivalent:
\begin{enumerate}
\item[(i)]   $\Phi_{f,A}^\infty(I)$ is invertible;
\item[(ii)] there exist $(T_1,\ldots, T_n)\in {\cV}_{f,\cP}^m(\cH)$ such that
$ (id-\Phi_{f,T})^m(I)= 0
$
and an invertible operator $Y\in B(\cH)$ such that
$$
A_i=Y^{-1} T_i Y,\qquad i=1,\ldots,n.
$$
\end{enumerate}
In the particular case when $\Phi_{f,A}(I)\leq I$,  the limit
$\text{\rm SOT-}\lim_{k\to \infty} \Phi_{f,A}^k(I)$   always exists.
\end{corollary}
\begin{proof} Assume that item (i) holds.
Since  $\Phi_{f,A}$ is  WOT-continuous on bounded sets (see Lemma
\ref{wot-cp}) and the limit $ \text{\rm SOT-}\lim_{k\to \infty}
\Phi_{f,A}^k(I)$ exists, we have
$\Phi_{f,A}(\Phi_{f,A}^\infty(I))=\Phi_{f,A}^\infty(I)$.
 Taking into account that
$\Phi_{f,A}^\infty(I)$ is invertible,   item (ii) follows from
Theorem \ref{simi}. Conversely,  assume that  item (ii) holds.  Then
Theorem \ref{simi}, implies $cI\leq \Phi_{f,A}^k(I)\leq dI$ for any
$k\in\NN$. Hence,  the operator $\Phi_{f,A}^\infty(I)$ is
invertible, and the proof is complete.
\end{proof}

Given  $A,B\in B(\cH)$ two self-adjoint  operators, we say that
$A<B$ if $B-A$ is positive and invertible, i.e., there exists a
constant $\gamma>0$ such that $\left<(B-A)h,h\right>\geq
\gamma\|h\|^2$ for any $h\in \cH$. Note that  $C\in B(\cH)$ is a
strict contraction ($\|C\|<1$) if and only if $C^*C<I$.

A version of Rota's model theorem (see \cite{R}, \cite{H1}) asserts
that any operator with spectral radius less than
 one is similar to a strict contraction. In what follows we present an
 analogue of this result  in our multivariable noncommutative
 setting.

\begin{theorem} \label{simi2}
 Let $m\geq 1$,  $f :=\sum_{|\alpha|\geq 1} a_\alpha X_\alpha$ be
a positive regular free holomorphic function and let $\cP$ be a family of
noncommutative
 polynomials. If  $A:=(A_1,\ldots,
A_n)\in B(\cH)^n$  is such that $\sum_{|\alpha|\geq 1} a_\alpha
A_\alpha A_\alpha^* $ is convergent in the weak operator topology
and $p(A_1,\ldots, A_n)=0$,\quad $p\in \cP$, then the following
statements are equivalent.
\begin{enumerate}
\item[(i)]
  There exist $T:=(T_1,\ldots, T_n) \in {\cV}_{f,\cP}^m(\cH)$ with
 $(id-\Phi_{f,T})^m(I)>0$
 and an invertible operator $Y\in B(\cH)$ such that
$$
A_i=Y^{-1} T_i Y,\qquad i=1,\ldots,n.
$$
 \item[(ii)] $\Phi_{f,A}$ is power bounded and there exists a positive operator $Q\in B(\cH)$
  such that
 $$(id-\Phi_{f,A})^m(Q)>0.
 $$
\item[(iii)] $r_f(A_1,\ldots,A_n)<1$.
  \item[(iv)] $\lim\limits_{k\to \infty} \|\Phi_{f,A}^k(I)\|=0$.
\item[(v)] $\Phi_{f,A}$ is power bounded and  there is an
 invertible positive operator $R\in B(\cH)$, such that
  the equation
  \begin{equation*}
  (id-\Phi_{f,A})^m(X)=R
  \end{equation*}
  has a positive  solution  $X$ in $B(\cH)$.
\end{enumerate}
 Moreover, in this case, for any
 invertible positive operator $R\in B(\cH)$,
  the equation
  $
  (id-\Phi_{f,A})^m(X)=R
  $
  has a
  unique  positive solution, namely,
  $$
  X:=\sum_{k=0}^\infty \left(\begin{matrix} k+m-1\\ m-1
\end{matrix}\right)\Phi_{f,A}^k(R),
  $$
  where the convergence is in the uniform topology, which is an invertible operator.
\end{theorem}
\begin{proof} First we prove the equivalence  (i) $\Leftrightarrow$ (ii).
Assume that (i) holds and   $(id-\Phi_{f,T})^m(I)\geq cI$ for some $c>0$. Then we have
$$
Y\left[(id-\Phi_{f,A})^m(Y^{-1} (Y^{-1})^*)\right] Y^*\geq cI.
$$
Setting $Q:=Y^{-1} (Y^{-1})^*$ we deduce that $(id-\Phi_{f,A})^m(Q)>0$. The fact
 that $\Phi_{f,A}$ is power bounded is due to Proposition \ref{properties}.

Conversely, assume that  item (ii) holds and let $Q\in B(\cH)$ be  a positive
operator  such that $(id-\Phi_{f,A})^m(Q)>0$. Since $\Phi_{f,A}$ is power bounded,
Lemma \ref{ineq-lim} implies $(id-\Phi_{f,A})^s(Q)\geq 0$, $s=1,\ldots,m$. On the other hand, since $\Phi_{f,A}$ is a positive linear map, we deduce that
$$
0<(id-\Phi_{f,A})^m(Q)\leq \cdots \leq (id-\Phi_{f,A})(Q)\leq Q.
$$
Therefore, $Q$ is an invertible positive operator.
 Since
\begin{equation*}
(id-\Phi_{f,A})^m(Q)\geq b I
\end{equation*}
 for some constant $b>0$, we can choose $c>0$ such that $bI\geq cQ$, and  deduce that
 $$Q^{-1/2}[(id-\Phi_{f,A})^m(Q)]Q^{-1/2}\geq cI.$$
 Setting $T_i:=Q^{-1/2} A_iQ^{1/2}$, $i=1,\ldots,n$, the latter inequality implies
 $(id-\Phi_{f,T})^m(I)>0$. Since $\Phi_{f,A}$ is power bounded, so is $\Phi_{f,T}$. As above, using again Lemma \ref{ineq-lim} we obtain $(id-\Phi_{f,T})^s(I)>0$,
 $s=1,2,\ldots,m$, which shows that $T\in {\bf D}_f^m(\cH)$.
 Since  $p(A_1,\ldots, A_n)=0$,\quad $p\in \cP$, we deduce that $T\in \cV^m_{f,\cP}(\cH)$.
Therefore, item (i) holds.

Now we prove the  equivalence (iii) $\Leftrightarrow$ (iv).
Assume that item  (iii) holds and let $a>0$ be such that  $ r( A_1,\ldots,
A_n)< a<1$. Then there is $m_0\in \NN$ such that $\|\Phi_{f,A}^k(I)\|\leq
a^k$ for any $k\geq m_0$. This clearly implies condition (iv).  Now, we  assume that (iv) holds.
 Note that
\begin{equation*}
\begin{split}
r_f( A_1,\ldots, A_n)^j&=\lim_{k\to
\infty}\left[\|\Phi_{f,A}^{jk}(I)\|^{1/2kj}\right]^j\\
&=\lim_{k\to \infty}
\|\Phi_{f,A}^{j(k-1)}(\Phi_{f,A}^j(I))\|^{1/2k}\\
&\leq \lim_{k\to
\infty}\left(\|\Phi_{f,A}^j(I)\|^k\right)^{1/2k}=\|\Phi_{f,A}^j(I)\|^{1/2}
\end{split}
\end{equation*}
 for any $j\in \NN$. Consequently,  $r_f( A_1,\ldots, A_n)<1$, so item (iii) holds.
The implication (v) $\Rightarrow$ (ii) is obvious. In what follows
we prove that
  (i) $\Rightarrow$ (iii).
  Assume that  there exists $T:=(T_1,\ldots, T_n) \in {\cV}_{f,\cP}^m(\cH)$ with
 $(id-\Phi_{f,T})^m(I)>0$
 and an invertible operator $Y\in B(\cH)$ such that
$$
A_i=Y^{-1} T_i Y,\qquad i=1,\ldots,n.
$$
Recall that under these conditions we have, in particular,  $\|\Phi_{f,A}(I)\|<1$.
  On the other hand, note that
\begin{equation*}
\begin{split} r_f(T_1,\ldots, T_n)&=
r_f(YA_1Y^{-1},\ldots,YA_nY^{-1})\\
&=\lim_{k\to\infty}\|\Phi_{f,YAY^{-1}}^k(I)\|^{1/2k}\\
&\leq \lim_{k\to\infty}\|Y\|^{1/k}\|\Phi_{f,A}^k(I)\|^{1/2k}\\
&=r_f(A_1,\ldots,A_n ).
\end{split}
\end{equation*}
  Hence, applying this
  inequality when $Y $ is replaced by its inverse,
  we  deduce that
\begin{equation*}
\begin{split}
r_f(A_1,\ldots,A_n
)&=r_f\left(Y^{-1}(YA_1Y^{-1})Y,\ldots,Y^{-1}(YA_nY^{-1})Y\right)\\
&\leq
r_f(YA_1Y^{-1},\ldots,YA_nY^{-1})=r_f(T_1,\ldots, T_n).
\end{split}
\end{equation*}
Therefore, we have
\begin{equation*}
\begin{split}
r_f(A_1,\ldots,A_n)&=r_f(T_1,\ldots, T_n) \\
& =\lim_{k\to\infty} \|\Phi_{f,T}^k(I)\|^{1/2k}\leq
\|\Phi_{f,T}(I)\|^{1/2}<1,
\end{split}
\end{equation*}
which shows that   item (iii) holds. Now,  we prove the implication
(iii)
  $\Rightarrow$ (v). To this end,  assume that $r_f(A_1,\ldots,A_n)<1$ and let $R\in B(\cH)$ be  an invertible positive operator.
   We have
    $$
    \frac {1} {\|R^{-1}\|} \, I \leq R\leq \sum_{k=0}^\infty\left(\begin{matrix} k+m-1\\ m-1
\end{matrix}\right) \Phi_{f,A}^k(R)
    \leq \left(\|R\| \sum_{k=0}^\infty \left(\begin{matrix} k+m-1\\ m-1
\end{matrix}\right)\|\Phi_{f,A}^k(I)\|\right)\, I.
    $$
 Note that
 $$
 \lim_{k\to\infty} \left[\left(\begin{matrix} k+m-1\\ m-1
\end{matrix}\right)\|\Phi_{f,A}^k(I)\|\right]^{1/2k}=r_f(T_1,\ldots,T_n)<1.
$$
Consequently,
   \begin{equation}\label{ab}
 aI \leq  \sum_{k=0}^\infty \left(\begin{matrix} k+m-1\\ m-1
\end{matrix}\right)\Phi_{f,A}^k(R) \leq bI
 \end{equation}
 for some  constants  $0<a<b$, where the convergence of the series is in the operator norm topology.
 Now, we can prove that
 $$
 (id-\Phi_{f,A})^m\left[ \sum_{k=0}^\infty \left(\begin{matrix} k+m-1\\ m-1
\end{matrix}\right)\Phi_{f,A}^k(R) \right]=R.
 $$
 Indeed, since $(id-\Phi_{f,A}) \Phi_{f,A}=\Phi_{f,A}(id-\Phi_{f,A})$, we can use Lemma \ref{ineq-lim} and the fact that
 $$0\leq \lim_{k\to \infty} \|\Phi_{f,A}^k(R)\|\leq \|R\| \lim_{k\to \infty} \|\Phi_{f,A}^k(I)\|=0$$
  to obtain
 \begin{equation*}\begin{split}
 (id-\Phi_{f,A})^m\left[ \sum_{k=0}^\infty \left(\begin{matrix} k+m-1\\ m-1
\end{matrix}\right)\Phi_{f,A}^k(R) \right]&=
 \sum_{k=0}^\infty \left(\begin{matrix} k+m-1\\ m-1
\end{matrix}\right) \Phi_{f,A}^k(id-\Phi_{f,A})^m(R )\\
&=
R- \text{\rm SOT-}\lim_{k\to \infty}\sum_{i=0}^{m-1}\left(\begin{matrix} k+i\\ i
\end{matrix}\right)\Phi_{f,A}^{k+1}(id-\Phi_{f,A})^i(R)\\
&= R-\text{\rm SOT-}\lim_{k\to \infty}  \Phi_{f,A}^k(R)=R
\end{split}
\end{equation*}
 Consequently, and due to  relation \eqref{ab},
  $$
  X:=\sum_{k=0}^\infty \left(\begin{matrix} k+m-1\\ m-1
\end{matrix}\right)\Phi_{f,A}^k(R)
  $$
  is an invertible positive solution of the equation  $(id-\Phi_{f,A})^m(X)=R$.
 Therefore item (v) holds.

   To prove the last part of the theorem,
   let $X'\geq 0$ be an invertible operator such
   $(id-\Phi_{f,A})^m(X')=R$, where $R\geq 0$ is a fixed   arbitrary invertible operator. Then, as above,  we have
   \begin{equation*}\begin{split}
  \sum_{k=0}^\infty \left(\begin{matrix} k+m-1\\ m-1
\end{matrix}\right)\Phi_{f,A}^k(R) &=
 \sum_{k=0}^\infty \left(\begin{matrix} k+m-1\\ m-1
\end{matrix}\right) \Phi_{f,A}^k(id-\Phi_{f,A})^m(X' )\\
&= X'-\text{\rm SOT-}\lim_{k\to \infty}  \Phi_{f,A}^k(X')=X'.
\end{split}
\end{equation*}

Here we used that $\|\Phi_{f,A}^k(X')\|\leq \|X'\| \|\Phi_{f,A}^k(I)\|\to 0,$
   as $k\to \infty$.
   Therefore  there is  unique  positive solution
   of the inequality $(id-\Phi_{f,A})^m(X)=R$.
   The proof is complete.
\end{proof}

Now we  can obtain the following
multivariable generalization of Rota's  similarity result  (see Paulsen's book \cite{Pa-book}).

\begin{corollary} \label{rota}
Under the hypotheses of Theorem \ref{simi2}, if the joint spectral
radius $r_f(A_1,\ldots,A_n)<1$, then the $n$-tuple
$$T:=(P^{-1/2} A_1
P^{1/2},\ldots, P^{-1/2} A_n P^{1/2})$$
   is in the noncommutative variety $\cV_{f,\cP}^m(\cH)$
and  $(id-\Phi_{f,T})^m(I)>0$, where
$$P:=\sum_{k=0}^\infty \left(\begin{matrix} k+m-1\\
m-1
\end{matrix}\right)\Phi_{f,A}^k(I)$$
 is convergent
in the operator  norm topology  and
 $$
 \|P^{1/2}\|\|P^{-1/2}\|\leq \left(\sum_{k=0}^\infty
\left(\begin{matrix} k+m-1\\ m-1
\end{matrix}\right)\|\Phi_{f,A}^k(I)\|\right)^{1/2}.
$$
In particular, if  $f$ is a positive regular noncommutative
polynomial, then $P$ is in the $C^*$-algebra generated by
$A_1,\ldots, A_n$ and the identity.
\end{corollary}
\begin{proof} Since
$$
 \lim_{k\to\infty} \left[\left(\begin{matrix} k+m-1\\ m-1
\end{matrix}\right)\|\Phi_{f,A}^k(I)\|\right]^{1/2k}=r_f(T_1,\ldots,T_n)<1,
$$ the series  $\sum_{k=0}^\infty
\left(\begin{matrix} k+m-1\\ m-1
\end{matrix}\right)\|\Phi_{f,A}^k(I)\|$  is convergent and  we have
$$I\leq P\leq\sum_{k=0}^\infty
\left(\begin{matrix} k+m-1\\ m-1
\end{matrix}\right)\|\Phi_{f,A}^k(I)\|,
$$
which implies the  upper bound estimation  for
$\|P^{1/2}\|\|P^{-1/2}\|$. A closer look at the proof of  Theorem
\ref{simi2} and taking  $R=I$ leads to the desired result. The last
part of this corollary is now obvious.
\end{proof}

Using  Theorem \ref{simi2} and Theorem \ref{pure-contr}, we deduce
the following result.

\begin{corollary} \label{simi3}
 Let $A:=(A_1,\ldots, A_n)\in B(\cH)^n$ be under the hypotheses of Theorem
 \ref{simi2}. Then the following statements hold.
 \begin{enumerate}
 \item[(i)] If \, $r_f(A_1,\ldots,A_n)=0$, then, for any
$\epsilon>0$,  $(A_1,\ldots, A_n)$ is jointly similar to an
$n$-tuple of operators $(T_1,\ldots, T_n)\in \epsilon \cV_{f,\cP}^m
(\cH)$.
 \item[(ii)]
 If there exist
 positive constants
 $0<a\leq b$   and a  positive operator
  $R\in B(\cH)$ such that
 \begin{equation*}
 aI \leq  \sum_{k=0}^\infty \left(\begin{matrix} k+m-1\\ m-1
\end{matrix}\right)\Phi_{f,A}^k(R) \leq bI,
 \end{equation*}
 then
   $(A_1,\ldots, A_n)$ is jointly similar to an
$n$-tuple of operators $(T_1,\ldots, T_n)\in \cV_{f,\cP}^m(\cH)$.
If, in addition, $R$ is invertible, then $(id-\Phi_{f,T})^m(I)>0$.
\end{enumerate}
\end{corollary}

The next result provides necessary and sufficient conditions for an $n$-tuple of operators
 be  similar
 to an $n$-tuple in the noncommutative variety $\cV^m_{f,\cP}$,
 $m\geq 1$.

\begin{theorem} \label{simi4} Let $m\geq 1$,  $f :=\sum_{|\alpha|\geq 1} a_\alpha X_\alpha$ be
a positive regular free holomorphic function and let $\cP$ be a family of
noncommutative
 polynomials. If  $A:=(A_1,\ldots,
A_n)\in B(\cH)^n$  is such that $\sum_{|\alpha|\geq 1} a_\alpha
A_\alpha A_\alpha^* $ is convergent in the weak operator topology
and $p(A_1,\ldots, A_n)=0$,\quad $p\in \cP$, then the following
statements are equivalent.
\begin{enumerate}
\item[(i)]
  There exist  an $n$-tuple $(T_1,\ldots, T_n)\in \cV^m_{f,\cP}(\cH)$
   and  an
invertible operator $Y\in B(\cH)$ such that
$$
A_i=Y^{-1} T_i Y,\qquad i=1,\ldots,n.
$$
\item[(ii)] $\Phi_{f,A}$ is power bounded and there is an invertible
positive operator $R\in B(\cH)$ such that $$(id-\Phi_{f,A})^m(R)\geq
0.
$$
\end{enumerate}
If, in addition, $m=1$ and   $\cP$ is a set of homogeneous
polynomials, then the statements above are equivalent to the
following:
\begin{enumerate}
\item[(iii)] the map $\Psi:\cA_n(\cV^1_{f,\cP})\to B(\cH)$ defined by
$$\Psi(p(B_1,\ldots, B_n)):=
 p(A_1,\ldots, A_n)$$  is completely bounded, where $\cA_n(\cV^1_{f,\cP})$ is the noncommutative
 variety  algebra.
\end{enumerate}
\end{theorem}
\begin{proof}
The proof of the  equivalence (i) $\Leftrightarrow$ (ii) is similar
to the proof of the same equivalence from Theorem \ref{simi2}.
Consider the case $m=1$. If item (i) holds, then
   $$
   p(A_1,\ldots, A_n)= Yp(T_1,\ldots, T_n)Y^{-1}
   $$
   for any noncommutative polynomial $p$.
    Using the noncommutative von Neumann inequality  for $\cV^1_{f,\cP}(\cH)$,
    we  deduce $
   \|\Psi\|_{cb}\leq\|Y\|\|Y^{-1}\|.
   $
   Now, if we assume that item (iii) holds, then using  Paulsen's  similarity result \cite{Pa}
   and the fact (see \cite{Po-domains}) that
   any completely contractive representation of the noncommutative variety algebra
   $\cA_n(\cV^1_{f,\cP})$ is generated  by an $n$-tuple
   $(T_1,\ldots, T_n)\in \cV^1_{f,\cP}(\cH)$, we infer that
   $(A_1,\ldots, A_n)$ is simultaneously similar to an $n$-tuple
   $(T_1,\ldots, T_n)\in \cV^1_{f,\cP}(\cH)$.
   The proof is complete.
\end{proof}

\bigskip

\section{ Joint invariant subspaces  and triangulations for $n$-tuples of  operators
  in  noncommutative varieties}

In this section,    we obtain Wold type decompositions  and  prove the existence of  triangulations of type
$$
\left(\begin{matrix}C_{\cdot 0}&0\\
*& C_{\cdot 1}\end{matrix} \right)\quad \text{ and }\quad \left(\begin{matrix}C_{c}&0\\
*& C_{cnc}\end{matrix} \right)
$$
for any $n$-tuple  of operators in the noncommutative variety
${\cV}^1_{f,\cP}(\cH)$. As consequences,  we   show that certain
classes of $n$-tuples  of  operators in ${\cV}^1_{f,\cP}(\cH)$  have
non-trivial joint invariant subspaces.

\begin{theorem} \label{inv-ker} Let  $f:=\sum_{|\alpha|\geq 1} a_\alpha X_\alpha$
 be   a positive regular  free holomorphic function and $m\geq 1$.
Let $A:=(A_1,\ldots, A_n)\in B(\cH)^n$ be such that
$\sum_{|\alpha|\geq 1} a_\alpha A_\alpha A_\alpha^* $ is convergent
in the weak operator topology  and $\Phi_{f,A}$ is power bounded. If
$D\in B(\cH)$ be a positive operator such that $(id-
\Phi_{f,A})^m(D)\geq 0$. Then the subspaces
$$\ker D, \quad \{h\in \cH: \ \lim_{k\to\infty}\Phi_{f,A}^k(D) h=0\},\quad \text{ and
} \quad  \{h\in \cH: \ \Phi_{f,A}^k(D)h=Dh \ \text{ for all }\
k\in \NN\}
$$
are invariant under each operator $A_i^*$, $i=1,\ldots,n$.

In particular, if $\cM$ is a subspace of $\cH$ and $(id-
\Phi_{f,A})^m(P_\cM)\geq 0$, where $P_\cM$ is the orthogonal
projection onto $\cM$,  then $\cM$ is invariant under each operator
$A_i$.
\end{theorem}
\begin{proof}

 Due  to Lemma \ref{ineq-lim}, we  have   $\Phi_{f, A}(D)\leq D$. Consequently, for any $h\in \ker D$,
$$
0\leq \sum_{|\alpha|\geq 1} \langle a_\alpha A_\alpha DA_\alpha^* h,
h\rangle \leq \langle Dh,h\rangle=0.
$$
Hence, $\|a_{g_i}D^{1/2} A_i^* h\|=0$ for  $i=1,\ldots, n$. Since
$a_{g_i}\neq 0$, we deduce that   $A_i^* h\in \ker D$. Therefore,
$\ker D$ is invariant under each operator $A_i^*$.
 Now, let
  $D= B+C$ be the canonical decomposition of $D$
 with respect to
$\Phi_{f,A}$.
 According to Theorem \ref{decomp}, we have
 $$
 C\geq 0,\quad (id-\Phi_{f,A})^m(C)\geq 0, \quad  \text{\rm
SOT-}\lim_{k\to \infty}  \Phi_{f,A}^k (C)=0,
$$
and
$$
B=\text{\rm SOT-}\lim_{k\to\infty} \Phi_{f,A}^k(D), \quad
\Phi_{f,A}(B)=B.
$$
Now, due to the first part of this theorem, applied to $B$ and $C$,
respectively,  the subspaces $\ker B$ and $\ker C $ are
 invariant under each  each  operator $A_i^*$, $i=1,\ldots, n$.
  Note that
 $$
 \ker B= \{h\in \cH: \  \lim_{k\to\infty} \Phi_{f,A}^k(D)h=0\}.
 $$
 Since $\Phi_{f,A}(D)\leq D$, it is easy to see that
 $$
 \ker C=\{ h\in \cH:\  \lim_{k\to\infty} \Phi_{f,A}^k(D)h= Dh\}
 =\{ h\in \cH:\  \Phi_{f,A}^k(D)h= Dh, \ k\in \NN\}.
 $$
 Taking $D:=P_\cM$, we obtain the last part of the theorem.
 The proof is complete.
\end{proof}

  An interesting  consequence of Proposition \ref{inv-ker} is the following.

 \begin{corollary}\label{inv2} Let $T:=(T_1,\ldots, T_n)\in
 {\bf D}^m_{f}(\cH)$  be such that $(id-\Phi_{f,T})^m(I)=I$ and let $\cM\subseteq \cH$ be a subspace.
   Then the following statements hold.
   \begin{enumerate}
   \item[(i)]
$\cM$ is   an invariant subspace  under each  operator  $T_i$, \
$i=1,\ldots, n$,  if and only if \,$\Phi_{f,T}(P_\cM)\leq P_\cM$.
\item[(ii)]
$\cM$ is reducing
 under each  operator $T_i$, \
$i=1,\ldots,n$, if and only if \,$\Phi_{f,T}(P_\cM)= P_\cM$.
\end{enumerate}
 \end{corollary}

\begin{proof}  Due to Theorem
\ref{inv-ker}, if $\Phi_{f,T}(P_\cM)\leq P_\cM$, then the subspace $
\cM$ is invariant under under each operator $T_i$, \ $i=1,\ldots,
n$. Conversely, assume $\cM$ is invariant under under each $T_i$, \
$i=1,\ldots, n$. Then $P_\cM^\perp T_i P_\cM^\perp=P_\cM^\perp T_i$,
where $P_\cM^\perp:=I-P_\cM$. As seen before in this paper, the
condition $(id-\Phi_{f,T})^m(I)=1$ implies $\Phi_{f,T}(I)=I$.
Consequently, we have
$$
\Phi_{f,T}(P_\cM^\perp)P_\cM^\perp=\Phi_{f,T}(I)P_\cM^\perp=P_\cM^\perp
=P_\cM^\perp\Phi_{f,T}(I)P_\cM^\perp=P_\cM^\perp\Phi_{f,T}(P_\cM^\perp)P_\cM^\perp.
$$
Since the operators $\Phi_{f,T}(P_\cM^\perp)$ and $I-P_\cM^\perp$
are positive and commuting, we have
$$
\Phi_{f,T}(P_\cM^\perp)-P_\cM^\perp=\Phi_{f,T}(P_\cM^\perp)(I-P_\cM^\perp)\geq
0.
$$
Consequently, $\Phi_{f,T}(P_\cM^\perp)\geq P_\cM^\perp$. Since
$\Phi_{f,T}(I)=I$, we infer that $\Phi_{f,T}(P_\cM)\leq P_\cM$.
Therefore, item (i) holds. To prove (ii), note that due to  part
(i), $\cM$ is reducing
 under each operator $T_i$, \
$i=1,\ldots,n$, if and only if $\Phi_{f,T}(P_\cM)\leq P_\cM$ and
  $\Phi_{f,T}(P_\cM^\perp)\leq P_\cM^\perp$. Since $\Phi_{f,T}(I)=I$, the result
  follows.
\end{proof}

Let
 $f $
 be   a positive regular  free holomorphic function, $m\geq 1$, and
 let $K^{(m)}_{f,T,R}$ be   the noncommutative Berezin kernel associated with the
noncommutative  domain  ${\bf D}_{f }^m$, i.e., associated with the
quadruple $q:=(f,m,T,R)$, where  $R:=(id-\Phi_{f,T})^m(I)$. We
remark that,  as in the proof of Theorem \ref{poisson}, one can use
Lemma \ref{lemma1} and  Lemma \ref{ineq-lim} to obtain  relation
$$\left(K^{(m)}_{f,T,R}\right)^* K^{(m)}_{f,T,R}=\sum_{k=0}^\infty \left(\begin{matrix} k+m-1\\ m-1
\end{matrix}\right)
\Phi_{f,T}^k(R)= I- Q_{f,T},
$$
where $Q_{f,T}:= \text{\rm SOT-}\lim_{k\to\infty}  \Phi_{f,T}^k(I)$.

\begin{lemma} \label{decomp2} Let
 $f:=\sum_{|\alpha|\geq 1} a_\alpha X_\alpha$
 be   a positive regular  free holomorphic function and $m\geq 1$.
 If $(T_1,\ldots, T_n)$ is an $n$-tuple of operators
  in the noncommutative  domain  ${\bf D}_{f }^m(\cH)$, then the limit
  $$
  Q_{f,T}:= \text{\rm SOT-}\lim_{k\to\infty}  \Phi_{f,T}^k(I)
 $$
 exists and
  we have
 \begin{equation*} \begin{split}
\ker Q_{f,T}&=\{ h\in \cH:\ \lim_{k\to\infty} \left<\Phi_{f,T}^k(I)h,h\right>=0\}\\
&=\{h\in \cH: \ \|K^{(m)}_{f,T,R} h\|=\|h\|\}\\
&=\ker \left[I-\left(K^{(m)}_{f,T,R}\right)^*K^{(m)}_{f,T,R}\right]
\end{split}
\end{equation*}
and
\begin{equation*}\begin{split}
\ker (I-Q_{f,T})&=\{ h\in \cH:\  \Phi_{f,T}^k(I)h =h, ~k\in \NN\}\\
&= \{ h\in \cH:\  \left<\Phi_{f,T}^k(I)h, h\right> =\|h\|^2, ~k\in \NN\}    \\
&=\ker K^{(m)}_{f,T,R},
\end{split}
\end{equation*}
where $K^{(m)}_{f,T,R}$ in the noncommutative Berezin kernel
associated with the noncommutative  domain  ${\bf D}_{f }^m$.
\end{lemma}
\begin{proof} Since $\Phi_{f,T}(I)\leq I$, the sequence of positive
operators $\Phi_{f,T}^k(I)$ is decreasing. Consequently,
 the operator
 $Q_{f,T}$ exists  and has the properties:  $0\leq  Q_{f,T}\leq I$
 and
$ \Phi_{f,T}(Q_{f,T})= Q_{f,T}$. Using relation
$\left(K^{(m)}_{f,T,R}\right)^* K^{(m)}_{f,T,R} = I- Q_{f,T}$ we
deduce some of the equalities above. The others are fairly easy.
\end{proof}

 Now we can obtain the following Wold type decomposition
 for  $n$-tuples  of operators in the noncommutative domain  ${\bf D}_{f }^m(\cH)$.
\begin{theorem}\label{wold1} Let
 $f:=\sum_{|\alpha|\geq 1} a_\alpha X_\alpha$
 be   a positive regular  free holomorphic function and $m\geq 1$.
 If $(T_1,\ldots, T_n)$ is an $n$-tuple of operators
  in the noncommutative domain  ${\bf D}_{f }^m(\cH)$ and
  $$
  Q_{f,T}:= \text{\rm SOT-}\lim_{k\to\infty}  \Phi_{f,T}^k(I),
 $$
 then the space
$\cH$ admits a decomposition of the form
\begin{equation*}
\cH=\cM\oplus  \ker Q_{f,T} \oplus \ker (I-Q_{f,T}),
\end{equation*}
where
the subspaces  $\ker Q_{f,T}$ and $\ker (I-Q_{f,T})$
 are invariant under each  operator $T^*_i$, \ $i=1,\ldots, n$.
 \end{theorem}
 \begin{proof}  According  to Lemma \ref{decomp2},   the operator
 $Q_{f,T}$ exists  and has the properties:  $0\leq  Q_{f,T}\leq I$
 and
$ \Phi_{f,T}(Q_{f,T})= Q_{f,T}$. Since
$$
\cH= \overline{ Q_{f,T}(\cH)}\oplus \ker Q_{f,T} \ \text{
and } \ \ker (I-Q_{f,T})\subseteq  Q_{f,T}(\cH),
$$
 we obtain  the  desired decomposition.
 The fact that $\ker Q_{f,T}$
 is  an invariant  subspace under each  operator $T^*_i$, \ $i=1,\ldots, n$,
 follows from Theorem \ref{inv-ker}.
Now we  assume that  $ Q_{f,T}\neq 0$.  According to Lemma
\ref{lemma1},
\begin{equation*}
K^{(m)}_{f,T,R}T_i^*=(W_i^*\otimes I_{\overline{
R^{1/2}(\cH)}})K^{(m)}_{f,T,R},\qquad i=1,\ldots,n.
\end{equation*}
Hence $\ker K^{(m)}_{f,T,R}$ is an invariant subspace  under  each
operator $T^*_i$, \ $i=1,\ldots, n$. On the other hand, due  to
Lemma \ref{decomp2}, we have $\ker K^{(m)}_{f,T,R}=\ker \left(I-
Q_{f,T}\right)$. The proof is complete.
 \end{proof}

We have another proof of the fact that $\ker (I-Q_{f,T})$
 are invariant under each  operator $T^*_i$, \ $i=1,\ldots, n$,
 which does not use
  the noncommutative Berezin kernel. Indeed, assume that  $ Q_{f,T}\neq
 0$.
Then
$$
\left<  Q_{f,T} h,h\right>=\langle
 \Phi_{f,T}^k( Q_{f,T})h,h\rangle \leq \|Q_{f,T}\|
\langle  \Phi_{f,T}^k (I)h,h\rangle, \qquad h\in \cH, k\in \NN.
$$
Taking the limit as $k\to\infty$, we obtain
$$
\left< Q_{f,T}h,h\right>\leq \| Q_{f,T}\|^2 \left< h,h\right>.
$$
Hence, $\| Q_{f,T}^{1/2}\|\leq \|Q_{f,T}^{1/2}\|^2 = \|
Q_{f,T}\|\leq 1$ and, consequently,  we deduce that  $\|
Q_{f,T}^{1/2}\|=0$ or $\|Q_{f,T}^{1/2}\|=1$. Since $ Q_{f,T}\neq 0$,
we must have  $\|Q_{f,T}\|=1$. We show now that
 the set $\ker (I-Q_{f,T})$ is invariant
under each $T_i^*$, \ $i=1,\ldots, n$.
 Indeed,
 note that $I-Q_{f,T}\geq 0$ and
 $$
 \Phi_{f,T}(I-Q_{f,T})= \Phi_{f,T}(I)-\Phi_{f,T}(Q_{f,T})\leq I-Q_{f,T}.
 $$
 Applying   Theorem \ref{inv-ker} to the positive operator
 $I-Q_{f,T}$, the result follows.

\smallskip

  Let $m\geq 1$,  $f:=\sum_{|\alpha|\geq 1} a_\alpha X_\alpha$ be
a positive regular free holomorphic function and let $\cP$ be a
family of noncommutative  polynomials with $\cN_\cP\neq 0$. Let
$(T_1,\ldots, T_n)$ be  an $n$-tuple of operators
  in the noncommutative variety $\cV_{f,\cP}^m(\cH)$ and  let $K_q$ be  the
   Berezin kernel associated with  $\cV_{f,\cP}^m(\cH)$, i.e.,
     associated with the tuple $q=(f,m,T,R,\cP)$, where
     $R:=(id-\Phi_{f,T})^m(I)$.
 Under these conditions, Lemma \ref{lemma2} and  Lemma \ref{ineq-lim}
 and  imply  $K_q^* K_q= I- Q_{f,T}$.
Consequently, one can
   obtain  the following version of Theorem \ref{wold1}.

\begin{corollary}\label{version2} The space $\cH$ admits an orthogonal  decomposition
\begin{equation*}
\cH=\cM\oplus  \ker (I- K_q ^*K_q) \oplus \ker K_q,
\end{equation*}
where
the subspaces $\ker (I- K_q ^*K_q)$ and $\ker K_q$
 are invariant under each  operator $T^*_i$, \ $i=1,\ldots, n$.
\end{corollary}

An interesting consequence of Theorem \ref{wold1} is the following
Wold type decomposition.

\begin{corollary}\label{wold2}  Let $m\geq 1$,  $f:=\sum_{|\alpha|\geq 1} a_\alpha X_\alpha$ be
a positive regular free holomorphic function and let $\cP$ be a
family of noncommutative  polynomials with $\cN_\cP\neq 0$. Let
$(T_1,\ldots, T_n)$ be  an $n$-tuple of operators
  in the noncommutative variety $\cV_{f,\cP}^m(\cH)$ and  let $K_q$ be  the
   Berezin kernel associated with  $\cV_{f,\cP}^m(\cH)$. Then the following statements are equivalent:
 \begin{enumerate}
 \item[(i)] the Hilbert space $\cH$ admits the orthogonal  decompositions
\begin{equation*}
\cH=  \ker  Q_{f,T}\oplus \ker (I- Q_{f,T})= \ker (I- K_q ^*K_q)
\oplus \ker K_q;
\end{equation*}
\item[(ii)] $Q_{f,T}$ is an orthogonal projection;
\item[(iii)] the noncommutative Berezin kernel $K_q$ is a partial isometry.
\end{enumerate}
 In this case,
the subspaces
$$\ker  Q_{f,T}= \ker (I- K_q ^*K_q)\quad \text{ and } \quad \ker (I-
Q_{f,T})=\ker K_q
$$
 are reducing for each  operator $T_i$, \ $i=1,\ldots, n$.
\end{corollary}
\begin{proof}
 Since $Q_{f,T}$ is a positive
operator,
  it is well-known  that
 $$
 \ker [Q_{f,T}-  Q_{f,T}^2]=\ker Q_{f,T}
 \oplus \ker (I-Q_{f,T}).
 $$
 On the other hand,  note  that $ \ker [Q_{f,T}-  Q_{f,T}^2]=\cH$
 if and only if
  $Q_{f,T}$ is an orthogonal projection.
   Using the results preceding this corollary, we can complete the proof.
\end{proof}

\smallskip

Let  $m=1$,  $p=\sum_{|\alpha|\geq 1} a_\alpha X_\alpha$ be a
positive regular noncommutative polynomial and let $\cP$ be a set of
noncommutative polynomials such that $1\in \cN_\cP$. In
\cite{Po-domains}, using standard theory of representations of
$C^*$-algebras,  we obtained the following Wold type decomposition
for non-degenerate
  $*$-representations of the unital $C^*$-algebra $C^*(B_1,\ldots,
B_n)$, generated by the   constrained weighted shifts associated
with  the noncommutative variety $\cV^1_{p,\cP}$,  and the identity.
  If  \
$\pi:C^*(B_1,\ldots, B_n)\to B(\cK)$ is  a non-degenerate
$*$-representation  of $C^*(B_1,\ldots, B_n)$ on a separable Hilbert
space  $\cK$, then $\pi$ decomposes into a direct sum
$$
\pi=\pi_0\oplus \pi_1 \  \text{ on  } \ \cK=\cK_0\oplus \cK_1,
$$
where $\pi_0$ and  $\pi_1$  are disjoint representations of
$C^*(B_1,\ldots, B_n)$ on the Hilbert spaces
\begin{equation*}
\begin{split}
\cK_0:&=\left\{ x\in \cK:\ \lim\limits_{k\to \infty}
\left< \Phi^k_{p, V}(I_\cK)x,x\right>=0\right\}\quad \text{ and } \\
 \cK_1:&=\left\{ x\in \cK:\ \left< \Phi^k_{p, V}(I_\cK)x,x\right>
=\|x\|^2 \ \text{ for any } k\in \NN\right\},
\end{split}
\end{equation*}
 respectively, where $V_i:=\pi(B_i)$, \ $i=1,\ldots, n$. Moreover, up to an isomorphism,
\begin{equation*}
\cK_0\simeq\cN_\cP\otimes \cG, \quad  \pi_0(X)=X\otimes I_\cG \quad
\text{ for } \  X\in C^*(B_1,\ldots, B_n),
\end{equation*}
 where $\cG$ is a Hilbert space with
$$
\dim \cG=\dim \left\{\text{\rm range}\, [
I_\cK-\Phi_{p,V}(I_\cK)]\right\},
$$
 and $\pi_1$ is a $*$-representation  which annihilates the compact operators   and
 $\Phi_{p,\pi_1(B)}(I_{\cK_1})=I_{\cK_1}$, where
 $\pi_1(B):=(\pi_1(B_1),\ldots, \pi_1(B_1)$. Moreover, the
 decomposition is essentially unique.

 Note that the decomposition above
 coincides with the one provided by Corollary \ref{wold2} when $(T_1,\ldots, T_n)=(V_1,\ldots, V_n)$.

\bigskip

We need a few more definitions.   Let $\cP$ be a set of
noncommutative polynomials.  We say that  an $n$-tuple of operators
$T:=(T_1,\ldots, T_n)\in {\cV}^1_{f,\cP}(\cH)$ is of class $C_{\cdot
0}$ (or pure ) if
$$
\lim_{k\to\infty}\left<\Phi_{f,T}^k(I) h,h\right>=0\quad \text{for
any }\quad h\in \cH,
$$
and of class $C_{\cdot 1}$ if
$$
\lim_{k\to\infty}\left<\Phi_{f,T}^k(I) h,h\right>\neq 0\quad
\text{for any }\quad h\in \cH, ~h\neq 0.
$$
 We say that  $T:=(T_1,\ldots, T_n)\in
{\cV}^1_{f,\cP}(\cH)$  has a triangulation of type $C_{\cdot 0}-
C_{\cdot 1}$ if there is an orthogonal decomposition
$\cH=\cH_0\oplus \cH_1$ with respect to which
$$
T_i=\left(\begin{matrix} C_i&0\\
*& D_i\end{matrix} \right),\qquad i=1,\ldots,n,
$$
and the entries have the following properties:
\begin{enumerate}
\item[(i)] $T_i^*\cH_0\subseteq \cH_0$ for any $i=1,\ldots,n$;
\item[(ii)] $(C_1,\ldots, C_n)\in {\cV}^1_{f,\cP}(\cH_0)$ is of class
$C_{\cdot 0}$;
\item[(iii)]
$(D_1,\ldots, D_n)\in {\cV}^1_{f,\cP}(\cH_1)$ is of class $C_{\cdot
1}$.
\end{enumerate}
The type of the entry denoted by $*$ is not specified.

\begin{theorem}\label{factori}
Every $n$-tuple  $T:=(T_1,\ldots, T_n)\in {\cV}^1_{f,\cP}(\cH)$ has
a triangulation of type
$$
\left(\begin{matrix}C_{\cdot 0}&0\\
*& C_{\cdot 1}\end{matrix}
\right)
$$
Moreover, this triangulation is uniquely determined.
\end{theorem}
\begin{proof}
First, note that due to Theorem \ref{wold1},  the subspace
$$
\cH_0:=\left\{ h\in\cH:\ \lim_{k\to\infty}\left<\Phi_{f,T}^k(I_\cH)
h,h\right>=0\right\}
$$
is invariant under each operator $T_i^*$, $i=1,\ldots,n$. The
decomposition $\cH=\cH_0\oplus \cH_1$, where $\cH_1:=\cH\ominus
\cH_0$, yields the triangulations
$$
T_i^*=\left(\begin{matrix} C_i^*&*\\
0& D_i^*\end{matrix} \right),\qquad i=1,\ldots,n,
$$
where $C_i^*:=T_i^*|_{\cH_0}$ and $D_i^*:=P_{\cH_1}T_i^*|_{\cH_1}$
for each $i=1,\ldots,n$. Since $T_i^*(\cH_0)\subseteq \cH_0$,
$i=1,\ldots,n$, we have
$$
\Phi_{f,C}(I_{\cH_0})=P_{\cH_0} \Phi_{f,T}(I_\cH)|_{\cH_0}\leq
I_{\cH_0}
$$
and $p(C_1,\ldots, C_n)=P_{\cH_0} p(T_1,\ldots, T_n)|_{\cH_0}=0$ for
any $p\in \cP$. Therefore, $(C_1,\ldots, C_n)\in
\cV^1_{f,\cP}(\cH_0)$. On the other hand, we have
$$
\lim_{k\to\infty}\left<\Phi_{f,C}^k(I_{\cH_0})
h,h\right>=\lim_{k\to\infty}\left<\Phi_{f,T}^k(I_\cH)
h,h\right>=0,\quad h\in \cH_0,
$$
which shows that  the $n$-tuple  $C:=(C_1,\ldots, C_n) $ is of class
$C_{\cdot 0}$. Now,  due to the fact that $T_i(\cH_1)\subseteq \cH_1$,
$i=1,\ldots, n$,  and $\Phi_{f,T}$  is a positive map, we have
$$
\Phi_{f,D}(I_{\cH_1})= P_{\cH_1}\Phi_{f,T}(P_{\cH_1})|_{\cH_1}\leq
P_{\cH_1}\Phi_{f,T}( I_\cH)|_{\cH_1} \leq I_{\cH_1}
$$
and  $p(D_1,\ldots, D_n)=  p(T_1,\ldots, T_n)|_{\cH_1}=0$ for any
$p\in \cP$.  Therefore,  $(D_1,\ldots, D_n)\in
\cV^1_{f,\cP}(\cH_1)$.
 We need to show that
$$
\lim_{k\to\infty} \left<\Phi_{f,D}^k(I_{\cH_1}) h,h\right>\neq 0
\quad \text{ for all }\quad h\in\cH_1, h\neq 0.
$$
   Taking into account that $\Phi_{f,T}^k(I) P_{\cH_0}\to 0$  strongly,
   as $k\to\infty$, $\|\Phi_{f,T}^k(I) P_{\cH_0}\|\leq 1$ for $k\in \NN$,
   and $\Phi_{f,T}$ is WOT -continuous on bounded sets, we deduce that
  \begin{equation*}
\lim\limits_{k\to\infty}
  \left<\Phi_{f,T}^q \left(\Phi_{f,T}^k(I) P_{\cH_0}\right)h, h'\right>=0, \qquad h,h'\in \cH,
\end{equation*}
for each $q\geq 1$.
Hence, using the fact that $Q_{f,T}:= \text{\rm
SOT-}\lim_{k\to\infty} \Phi_{f,T}^k(I)$,   we have
\begin{equation}
\begin{split}\label{qft}
 \left<Q_{f,T}h, h'\right>
 &=\lim\limits_{k\to\infty}
  \left<\Phi_{f,T}^q \left(\Phi_{f,T}^k(I)\right))h, h'\right>\\
  &=\lim\limits_{k\to\infty}
  \left<\Phi_{f,T}^q \left(\Phi_{f,T}^k(I) P_{\cH_0}\right)h, h'\right> +\lim\limits_{k\to\infty}
  \left<\Phi_{f,T}^q \left(\Phi_{f,T}^k(I) P_{\cH_1}\right)h, h'\right>\\
  &=\lim\limits_{k\to\infty}
  \left<\Phi_{f,T}^q \left(\Phi_{f,T}^k(I) P_{\cH_1}\right)h, h'\right>
=\left<\Phi_{f,T}^q \left( Q_{f,T} P_{\cH_1}\right)h, h'\right>
\end{split}
\end{equation}
for any $h,h'\in \cH$.
Now, we need to prove  that
\begin{equation*}
\begin{split}
 \|\ Q_{f,T} h\|\leq \left<\Phi_{f,T}^q(P_{\cH_1})h, h\right>^{1/2},\qquad h\in \cH.
\end{split}
\end{equation*}
First, recall that  $\|\Phi_{f,T}^k(Q_{f,T} P_{\cH_1})\|\leq 1$, $k\in  \NN$, and
$\Phi_{f,T}$ is WOT -continuous on bounded sets. Consequently, given $h,h'\in \cH$,  the expression
$\left< \Phi_{f,T}^k(Q_{f,T} P_{\cH_1})h, h'\right>$ can be approximated by sums of type
$$
\Sigma :=\sum_{|\alpha_q|\leq N_q}\cdots \sum_{|\alpha_1|\leq N_1}
\left< a_{\alpha_q}\cdots a_{\alpha_1} T_{\alpha_q}\cdots T_{\alpha_1} (Q_{f,T} P_{\cH_1}) T_{\alpha_1}^*\cdots T_{\alpha_q}^* h, h'\right>,
$$
where $N_1,\ldots, N_q\in \NN$.
Since $\|Q_{f,T}\|\leq 1$, $a_\alpha\geq 0$, we obtain
\begin{equation*}
\begin{split}
&\left|\left< a_{\alpha_q}\cdots a_{\alpha_1} T_{\alpha_q}\cdots T_{\alpha_1} (Q_{f,T} P_{\cH_1}) T_{\alpha_1}^*\cdots T_{\alpha_q}^* h, h'\right>\right|\\
&\qquad \leq a_{\alpha_q}\cdots a_{\alpha_1}\left\| (Q_{f,T} P_{\cH_1}) T_{\alpha_1}^*\cdots T_{\alpha_q}^* h\right\|\left\|   T_{\alpha_1}^*\cdots T_{\alpha_q}^* h'\right\|\\
&\qquad \leq a_{\alpha_q}\cdots a_{\alpha_1} \left\|  P_{\cH_1} T_{\alpha_1}^*\cdots T_{\alpha_q}^* h\right\|\left\|   T_{\alpha_1}^*\cdots T_{\alpha_q}^* h'\right\|.
\end{split}
\end{equation*}
Applying  Cauchy's inequality, we get
\begin{equation*}
\begin{split}
|\Sigma|&\leq \left(\sum_{|\alpha_q|\leq N_q}\cdots \sum_{|\alpha_1|\leq N_1}
\left< a_{\alpha_q}\cdots a_{\alpha_1} T_{\alpha_q}\cdots T_{\alpha_1} ( P_{\cH_1}) T_{\alpha_1}^*\cdots T_{\alpha_q}^* h, h\right>\right)^{1/2}\\
&\qquad \qquad \times \left(\sum_{|\alpha_q|\leq N_q}\cdots \sum_{|\alpha_1|\leq N_1}
\left< a_{\alpha_q}\cdots a_{\alpha_1} T_{\alpha_q}\cdots T_{\alpha_1} T_{\alpha_1}^*\cdots T_{\alpha_q}^* h', h'\right>\right)^{1/2}.
\end{split}
\end{equation*}
Taking the limits as $N_1\to\infty,\ldots, N_q\to\infty$, we obtain
\begin{equation*}
\begin{split}
\left|\left< \Phi_{f,T}^q(Q_{f,T} P_{\cH_1})h,h'\right>\right|
&\leq \left< \Phi_{f,T}^q(P_{\cH_1})h,h\right>^{1/2} \left< \Phi_{f,T}^q(I)h',h'\right>^{1/2}\\
&\leq \left< \Phi_{f,T}^q(P_{\cH_1})h,h\right>^{1/2}\|h'\|
\end{split}
\end{equation*}
for any $h,h'\in \cH$.
Hence, we deduce that
$$
\left\| \Phi_{f,T}^q(Q_{f,T} P_{\cH_1})h\right\|\leq
 \left< \Phi_{f,T}^q(P_{\cH_1})h,h\right>^{1/2},\qquad h\in \cH, q\in \NN.
$$
Combining this inequality with  relation \eqref{qft}, we obtain
$$
\| Q_{f,T} h\|\leq
 \left< \Phi_{f,T}^q(P_{\cH_1})h,h\right>^{1/2}=\left<\Phi_{f,D}^q(I_{\cH_1})h, h\right>^{1/2},
 \qquad h\in \cH, q\in \NN.
$$
Let $h\in \cH_1$, $h\neq 0$, and assume that
$ \Phi_{f,D}^q(I_{\cH_1})h\to 0$, as $q\to\infty$. The above
inequality shows that $ Q_{f,T}h= 0$, i.e.,  $h\in \cH_0$, which is
a contradiction.

Now, we prove the uniqueness. Assume that there is another
decomposition $\cH=\cM_0\oplus \cM_1$ which yields the
triangulations
$$
T_i=\left(\begin{matrix} {E_i}&0\\
*& {F_i}\end{matrix} \right),\qquad i=1,\ldots,n,
$$
  of type
$
\left(\begin{matrix}C_{\cdot 0}&0\\
*& C_{\cdot 1}\end{matrix}
\right), $
where ${E_i}^*:=T_i^*|_{\cM_0}$ and
${F_i}^*:=P_{\cM_1}T_i^*|_{\cM_1}$ for each $i=1,\ldots,n$. To prove
uniqueness, it is enough to show that $\cH_0=\cM_0$. Notice that if
$h\in \cM_0$, then, due to the fact that  $(E_1,\ldots, E_n)$ is of
class $C_{\cdot 0}$, we have
$$
\lim_{k\to\infty} \left<\Phi_{f,T}^k h,h\right>= \lim_{k\to\infty}
\left<\Phi_{f,E}^k h,h\right>=0.
$$
Hence, $h\in \cH_0$, which proves that $\cM_0\subseteq \cH_0$.
Assume now that $h\in\cH_0\ominus \cM_0$.
Since $h\in\cM_1$, we have
$$
 \lim_{k\to\infty} \left<\Phi_{f,F}^k h,h\right> =\lim_{k\to\infty}
\left<\Phi_{f,T}^k (P_{\cM_1})h,h\right>\leq \lim_{k\to\infty}
\left<\Phi_{f,T}^k (I)h,h\right> =0.
$$
 Consequently,  since    $(F_1,\ldots, F_n)$ is of class $C_{\cdot 1}$,
  we must have $h=0$. Hence, we deduce that $\cH_0\ominus\cM_0=\{0\}$,
  which shows that $\cM_0=\cH_0$. This completes the proof.
\end{proof}

\begin{corollary}  If  $T:=(T_1,\ldots, T_n)\in {\cV}^1_{f,\cP}(\cH)$ is
 such  that $T\notin C_{\cdot 0}$ and $T\notin C_{\cdot 1}$, then there
 is a non-trivial joint invariant subspace
under  the operators $T_1,\ldots, T_n$.
\end{corollary}

 We say that  $T:=(T_1,\ldots, T_n)\in
{\cV}^1_{f,\cP}(\cH)$ is of class $C_{c}$  if
$$
\left<\Phi_{f,T}^k(I) h,h\right>=\|h\|^2\quad \text{for any }\quad
h\in \cH, k\in \NN,
$$
and of class $C_{cnc}$ if for each  $ h\in \cH, ~h\neq 0$, there exists $k\in \NN$ such that
$$
\left<\Phi_{f,T}^k(I) h,h\right>\neq \|h\|^2.
$$
We say that $T:=(T_1,\ldots, T_n)\in {\cV}^1_{f,\cP}(\cH)$  has a
triangulation of type $C_{c}- C_{cnc}$ if there is an orthogonal
decomposition $\cH=\cH_c\oplus \cH_{cnc}$ with respect to which
$$
T_i=\left(\begin{matrix} C_i&0\\
*& D_i\end{matrix} \right),\qquad i=1,\ldots,n,
$$
and the entries have the following properties:
\begin{enumerate}
\item[(i)] $T_i^*\cH_c\subseteq \cH_c$ for any $i=1,\ldots,n$;
\item[(ii)] $(C_1,\ldots, C_n)\in {\cV}^1_{f,\cP}(\cH_c)$ is of class
$C_{u}$;
\item[(iii)]
$(D_1,\ldots, D_n)\in {\cV}^1_{f,\cP}(\cH_{cnc})$ is of class $C_{
cnc}$.
\end{enumerate}

\begin{theorem}\label{factori2}  Let $\cP$ be a set of noncommutative polynomials.
Every $n$-tuple  of operators  $T:=(T_1,\ldots, T_n)\in
{\cV}^1_{f,\cP}(\cH)$ has a triangulation of type
$$
\left(\begin{matrix}C_{c}&0\\
*& C_{cnc}\end{matrix} \right).
$$
Moreover, this triangulation is uniquely determined.
\end{theorem}
\begin{proof}

Consider the subspace $\cH_c\subseteq \cH$ defined by
$$
\cH_c:=\left\{h\in \cH:\ \left<\Phi_{f,T}^k h,h\right>=\|h\|^2
\text{ for any } k\in \NN\right\}.
$$
 The fact  that   $\cH_c$ is invariant under each operator $T_1^*,\ldots,
 T_n^*$ is due
  to Lemma \ref{decomp2} and Theorem \ref{wold1}.
Consequently, we have the following triangulation with respect to
the decomposition $\cH=\cH_c\oplus \cH_{cnc}$,
$$
T_i=\left(\begin{matrix}C_i&0\\
*&D_i
\end{matrix}\right),\qquad i=1,\ldots,n,
$$
where $C_i^*:=T_i^*|_{\cH_c}$ and
$D_i^*:=P_{\cH_{cnc}}T_i^*|_{\cH_{cnc}}$ for each $i=1,\ldots,n$. As
in the proof of Theorem \ref{factori}, taking into account that
$T_i^*(\cH_c)\subseteq \cH_c$ and $T_i(\cH_{cnc})\subseteq
\cH_{cnc}$ for each $i=1,\ldots,n$,  we can show that $(C_1,\ldots,
C_n)\in \cV^1_{f,\cP}(\cH_c)$ and $(D_1,\ldots, D_n)\in
\cV^1_{f,\cP}(\cH_{cnc})$. Since
$$
 \left<\Phi_{f,C}^k(I_{\cH_c})
h,h\right>= \left<\Phi_{f,T}^k(I_\cH) h,h\right>=\|h\|^2,\qquad h\in
\cH_c, k\in \NN,
$$
  the $n$-tuple  $(C_1,\ldots, C_n) $ is of class
$C_{u}$. Now, we need to show that $(D_1,\ldots, D_n) $ is of class
$C_{cnc}$. To this end, let $h\in \cH_{cnc}$, $h\neq 0$, and assume
that $\left<\Phi_{f,D}^k(I_\cH) h,h\right>=\|h\|^2$ for all $k\in
\NN$. Then, we have
 \begin{equation*}
\begin{split}
\|h\|^2&=\left<\Phi_{f,D}^k(I_\cH)
h,h\right>=\left<\Phi_{f,T}^k(P_{\cH_{cnc}}) h,h\right>\\
&\leq \left<\Phi_{f,T}^k(I_\cH) h,h\right>\leq \|h\|^2.
\end{split}
\end{equation*}
Consequently, $\left<\Phi_{f,T}^k(I_\cH) h,h\right>= \|h\|^2$ for
all $k\in \NN$. Since $h\in \cH_{cnc}$, we must have $h=0$. This
proves that $(D_1,\ldots, D_n) $ is of class $C_{cnc}$. The
uniqueness of the triangulation can be proved as in Theorem
\ref{factori}. We leave it to the reader. The proof is complete.
\end{proof}

\begin{corollary}
If $T:=(T_1,\ldots, T_n)\in {\cV}^1_{f,\cP}(\cH)$  is such that $
\Phi_{f,T}(I)\neq I$ and there is a non-zero vector $h\in \cH$ such
that $ \left<\Phi_{f,T}^k h,h\right>=\|h\|^2$ for any $k\in \NN$,
then there is a non-trivial invariant subspace under the operators
$T_1,\ldots, T_n$.
\end{corollary}

  Note that $C_c\subseteq C_{\cdot 1}$. Combining
  Theorem \ref{factori}  with   Theorem \ref{factori2}, we
obtain another triangulation for    $n$-tuples  of operators in
${\cV}^1_{f,\cP}(\cH)$, that is,
$$
\left(
\begin{matrix}
C_{\cdot 0} &0&0\\
*& C_c &0\\
*&*& C_{cnc} \cap C_{\cdot 1}
\end{matrix}
\right).
$$

\smallskip

According to Corollary \ref{Fo-BR},   we have an analogue  of Foia\c
s \cite{Fo} and de Branges--Rovnyak \cite{BR}
  model theorem,  for   $n$-tuples of operators $(T_1,\ldots, T_n) \in \cV^m_{f,\cP}(\cH)$
  of class  $ C_{\cdot 0}$.
When $(A_1,\ldots, A_n)$ is of class $ C_{\cdot 1}$, i.e.,
$$
\lim_{k\to\infty}\left<\Phi_{f,A}^k(I) h,h\right>\neq 0\quad
\text{for any }\quad h\in \cH, ~h\neq 0,
$$
  we
can prove the following result.

\begin{theorem} Let $p:=\sum_{1\leq |\alpha|\leq N} a_\alpha X_\alpha$ be
 a positive regular  polynomial and let $\cP$ be a set
 of noncommutative polynomials.
If $A:=(A_1,\ldots, A_n)\in B(\cH)^n$ is
  an $n$-tuple of
operators   of class  $ C_{\cdot 1}$ such that $\Phi_{p,A} $ is
power bounded and $q(A_1,\ldots, A_n)=0$ for all $q\in \cP$, then
there exists $(T_1,\ldots, T_n)\in \cV^m_{p,\cP}(\cH)$  of class
$C_c$ such that
$$  A_iY=  YT_i ,\qquad i=1,\ldots,n, $$
 for some one-to-one
operator $Y\in B(\cH)$ with range dense in $\cH$. If, in addition,
$\cH$ is finite dimensional, then $Y$ is an invertible operator.
\end{theorem}
\begin{proof}
Since $\Phi_{p,A}$ is power bounded, there is $M>0$ such that
$\|\Phi_{p,A}^k\|\leq M$ for all $k\in \NN$. Note that for each
$h\in \cH$ with $h\neq 0$, we have
$$
\gamma_h:=\inf_{k\in \NN} \left< \Phi_{p,A}^k(I)h,h\right> >0.
$$
Indeed, if we assume that $\gamma_h=0$, then for any $\epsilon>0$
there is $k_0\in \NN$ such that $ \left<
\Phi_{p,A}^{k_0}(I)h,h\right> \leq \frac{\epsilon}{M}$. Since
$\Phi_{p,A}$ is a positive map, we have
$$
 \left< \Phi_{f,A}^{q+k_0}(I)h,h\right>
 \leq \|\Phi_{p,A}^q\|\left< \Phi_{p,A}^{k_0}(I)h,h\right>\leq
 \epsilon\qquad \text{  for any } q\in \NN.
 $$
 Consequently,   $\lim_{k\to \infty} \left< \Phi_{p,A}^k(I)h,h\right>=0$,
  which is a contradiction with the hypothesis.
Now, define
  $$
  [h,h']:=\operatornamewithlimits{LIM}_{k\to\infty} \left<
  \Phi_{p,A}^k(I)h,h'\right>, \qquad  h,h'\in \cH,
  $$
where LIM is a Banach limit. Due to the  properties of a Banach
limit, we have
$$
0< \gamma_h\leq [h,h] \leq M\|h\|^2,\qquad h\in \cH, h\neq 0,
$$
and
\begin{equation*}
\begin{split}
[h,h]&=\operatornamewithlimits{LIM}_{k\to\infty} \left<
  \Phi_{p,A}^{k+1}(I)h,h\right>
=\operatornamewithlimits{LIM}_{k\to\infty}\sum_{1\leq |\alpha|\leq
N} a_\alpha\left<
  \Phi_{p,A}^{k}(I)A_\alpha^*h, A_\alpha^*h\right>\\
  &=\sum_{1\leq |\alpha|\leq N} a_\alpha[A_\alpha^*h, A_\alpha^*h]
\end{split}
\end{equation*}
for any $h\in \cH$.  Using standard theory of bounded Hermitian
bilinear maps, we find  a  self-adjoint bounded operator $S\in
B(\cH)$ such that $[h,h']=\left<Sh,h'\right>$ for all $h,h'\in \cH$.
Therefore,  we have
$$
0<\gamma_h\leq \left<Sh,h\right>\leq M\|h\|^2, \qquad h\in \cH,
h\neq 0,
$$
which shows that $S$ is a one-to-one positive operator with range
dense   in $\cH$. Taking into account  the relations obtained above,
we deduce that
\begin{equation} \label{Sh}
\begin{split}
 \left<Sh,h\right>&=[h,h]=\sum_{1\leq |\alpha|\leq N} a_\alpha[A_\alpha^*h,
 A_\alpha^*h]\\
 &=\sum_{1\leq |\alpha|\leq N} a_\alpha\left<SA_\alpha^*h,
 A_\alpha^*h\right>=\left<\Phi_{p,A}(S)h,h\right>
\end{split}
\end{equation}
for any $h\in \cH$.  Hence, $\Phi_{p,A}(S)=S$ and $ a_{g_i}\|S^{1/2}
A_i^*h\|^2\leq \|S^{1/2} h\|^2$, $h\in \cH, $ for each $i=1,\ldots,
n$. Since $p$ is a positive regular polynomial,  we have $a_{g_i}>0$
and, consequently, it makes sense to define $Z_i:S^{1/2}(\cH)\to
\cH$ by setting
$$
Z_i(S^{1/2}h):=S^{1/2} A_i^*h,\qquad h\in \cH.
$$
Since $\|Z_i(S^{1/2}h)\|\leq \frac{1}{a_{g_i}}\|S^{1/2} h\|$,  $h\in
\cH$, and $S^{1/2}$ has range dense in $\cH$, $Z_i$ has a unique
bounded linear  extension to $\cH$, which we also denote  by $Z_i$.
Therefore, $\|Z_ix\|\leq  \frac{1}{a_{g_i}}\| x\|$, $x\in \cH$. Due
to relation \eqref{Sh}, we have
$$
\sum_{1\leq |\alpha|\leq N} a_\alpha\left<Z_{\tilde\alpha}^*
Z_{\tilde\alpha} S^{1/2}h, S^{1/2}h\right>=\|S^{1/2} h\|^2, \quad
h\in \cH.
$$
Since $S^{1/2}$ has range dense in $\cH$ and  $Z_i$  are bounded
operators on $\cH$, we deduce that  $\Phi_{p, Z^*}(I)=I$. Setting
now, $T_i:=Z_i^*$, $i=1,\ldots,n$, and $Y:=S^{1/2}$, we get
$\Phi_{p,T}(I)=I$ and $A_iY= YT_i$, $i=1,\ldots,n$. Note also that $
Yq(T_1,\ldots,T_n)=q(A_1,\ldots, A_n) Y=0$ for all $q\in \cP$. Since
$Y$ is one-to-one, we deduce that $q(T_1,\ldots,T_n)=0$. Therefore,
$(T_1,\ldots, T_n)\in \cV^m_{p,\cP}(\cH)$ is of class $C_c$. The
proof is complete.
\end{proof}




\bigskip


\begin{thebibliography}{99}




\bibitem{Ag2} {\sc J.~Agler},
Hypercontractions and subnormality, {\it J. Operator Theory} {\bf
13} (1985), 203--217.



\bibitem{Be} {\sc F.A.~Berezin},
Covariant and contravariant symbols of operators, (Russian), {\it
Izv. Akad. Nauk. SSSR Ser. Mat.} {\bf 36} (1972), 1134--1167.




\bibitem{Bu} {\sc J.W.~Bunce},
%
 Models for n-tuples of noncommuting operators,
 {\it  J. Funct. Anal.}  {\bf 57} (1984),  21--30.




\bibitem{CF1} {\sc G.~Cassier and T.~Fack},
 Contractions in von Neumann algebras,
 {\rm J. Funct. Anal.}
   {\bf 135} (1996), 297--338.



\bibitem{CF2} {\sc G.~Cassier and T.~Fack},
 On power bounded operators in finite von Neumann algebras,
{\rm  J. Funct. Anal.}
  {\bf  141} (1996),
 133--158.






 \bibitem{BR} {\sc L.~de Branges, J.~ Rovnyak},
   Canonical models in quantum scattering theory,
   {\it Perturbation Theory and its Applications in Quantum
   Mechanics}, 1966,
     pp. 295--392 Wiley, New York.



  \bibitem{Do}  {\sc R.G.~Douglas}, On the operator equation $S^* XT=X$
  and related topics,
 %
{\it Acta. Sci. Math. (Szeged)} {\bf 30} (1969), 19--32.



\bibitem{DFS}  {\sc R.G.~Douglas, C.~Foias, J.~Sarkar}, Resolutions of Hilbert spaces and similarity, preprint.



\bibitem{Fo} {\sc C.~Foia\c s}, A remark on the universal model for contractions of G. C. Rota.
 (Romanian) {\it Com. Acad. R. P. Rom\^ine} {\bf  13} 1963,  349--352.

\bibitem{H1} {\sc P.R.~Halmos},
{\it  A Hilbert space problem book}, Van Nostrand (New York, 1967).

\bibitem{H2} {\sc P.R.~Halmos}, Ten problems in Hilbert space,
{\it  Bull. Amer. Math. Soc.} {\bf 76} 1970 887--933.






\bibitem{O2} {\sc A.~Olofsson},  A characteristic operator function for the class of $n$-hypercontractions,
 {\it J. Funct. Anal.} {\bf 236} (2006), no. 2, 517--545.

\bibitem{O1} {\sc A.~Olofsson}, An operator-valued Berezin transform and the class of $n$-hypercontractions,
{\it Integral Equations Operator Theory} {\bf 58} (2007), no. 4,
503--549.

\bibitem{Pa} {\sc V.I.~Paulsen}, Every completely polynomially bounded
operator is similar to a contraction, {\it J. Funct. Anal.}, {\bf55}
(1984), 1--17.

\bibitem{Pa-book} {\sc V.I.~Paulsen},
 {\it Completely Bounded Maps and Dilations},
Pitman Research Notes in Mathematics, Vol.146, New York, 1986.





\bibitem{Pi} {\sc G.~Pisier}, A polynomially bounded operator on Hilbert space which is not
similar to a contraction, {\it  J. Amer. Math. Soc.} {\bf 10}
(1997), no. 2, 351--369.


\bibitem{Pi-book} {\sc G.~Pisier}, {\it Similarity problems and completely bounded
maps},
 Second, expanded edition. Includes the solution to ``The Halmos problem''.
 Lecture Notes in Mathematics, 1618. Springer-Verlag, Berlin, 2001. viii+198 pp.


\bibitem{Po-models} {\sc G.~Popescu}, Models for infinite sequences of
noncommuting operators, {\it Acta. Sci. Math. (Szeged)} {\bf 53}
(1989), 355--368.







\bibitem{Po-von} {\sc G.~Popescu},
 Von Neumann inequality for $(B(H)^n)_1$, {\it  Math. Scand.}
%
{\bf 68} (1991), 292--304.


\bibitem{Po-analytic} {\sc G.~Popescu},
 Multi-analytic operators on Fock spaces, {\it  Math. Ann.}
{\bf 303} (1995), 31--46.









\bibitem{Po-poisson} {\sc  G.~Popescu},
 Poisson transforms on some $C^*$-algebras generated by isometries,
 {\it J. Funct. Anal.}
  {\bf 161} (1999),  27--61





   \bibitem{Po-similarity} {\sc G.~Popescu},
    Similarity and ergodic theory of positive linear maps,
   {\it J. Reine Angew. Math.}
   {\bf 561} (2003), 87--129.



      \bibitem{Po-charact2} {\sc G.~Popescu},
      Characteristic functions and joint invariant subspaces,
      {\it J. Funct. Anal.}  {\bf 237} (2006), no. 1,  277--320.



\bibitem{Po-varieties} {\sc  G.~Popescu},
Operator theory on noncommutative varieties, {\it Indiana Univ.
Math.~J.} {\bf 55}, No.2, (2006), 389--442.



 \bibitem{Po-Berezin} {\sc G.~Popescu},
Noncommutative Berezin transforms and multivariable operator model
theory, {\it J.  Funct.  Anal. } {\bf  254} (2008), 1003-1057.

\bibitem{Po-unitary} {\sc  G.~Popescu},
 Unitary invariants in multivariable operator theory,  {\it Mem. Amer. Math.
Soc.} {\bf 200} (2009), no. 941,  vi+83 pp.


\bibitem{Po-pluriharmonic} {\sc G.~Popescu},
{Noncommutative transforms and free pluriharmonic functions},
 {\it Adv. Math.} {\bf 220} (2009), 831-893.


\bibitem{Po-domains} {\sc G.~Popescu},
Operator theory on noncommutative domains, {\it Mem.  Amer. Math.
Soc.}  {\bf 205} (2010), no. 964, vi+124 pp.






\bibitem{R} {\sc G.C.~Rota}, On models for linear operators,
{\it Comm. Pure Appl. Math} {bf 13} (1960), 469--472.


\bibitem{SzN} {\sc B.~Sz.-Nagy},
On uniformly bounded linear transformations in Hilbert space {\it
Acta. Sci. Math. (Szeged)} {\bf 11} (1947), 152--157.


\bibitem{SzF-book} {\sc B.~Sz.-Nagy and C.~Foia\c{s}}, {\it Harmonic
Analysis of Operators on Hilbert Space}, North Holland, New York 1970.



\bibitem{SzF1} {\sc B.Sz.-Nagy and  C.~Foia\c{s}},
 Toeplitz type operators and hyponormality,
{\it Operator theory: Adv. Appl.}
   {\bf 11}  (1983), 371--378.




 \bibitem{von}  {\sc J.~von Neumann},
      {Eine Spectraltheorie f\"ur allgemeine Operatoren eines unit\"aren
      Raumes,}
      {\it Math. Nachr.} {\bf 4} (1951), 258--281.





\end{thebibliography}
\end{document}